\documentclass[12pt,a4paper,english,reqno]{amsart}
\usepackage[english]{babel}
\usepackage{graphicx}
\usepackage{dsfont}
\usepackage[bf]{caption}
\usepackage{amsthm}
\usepackage{amsmath}
\usepackage{amssymb}

\newtheorem{thm}{Theorem}[section]
\newtheorem{cor}[thm]{Corollary}
\newtheorem{lem}[thm]{Lemma}
\newtheorem{prop}[thm]{Proposition}

\newtheorem{remark}[thm]{Remark}
\newtheorem{fact}[thm]{Fact}

\theoremstyle{rmk}

\newtheorem{rmk}[thm]{Remark}

\newcommand{\e} {\varepsilon}
\newcommand{\p} {\textnormal{\textsf{P}}}
\newcommand{\E} { \textnormal{\textsf{E}}}
\newcommand{\PP} { \mathbb{P} }
\newcommand{\EE} { \mathbb{E} }

\def\un{\mathds{1}}

\newcommand{\N} { \mathbb{N} }

\newcommand{\Z} { \mathbb{Z} }
\newcommand{\R} { \mathbb{R} }

\renewcommand{\P} { \mathbb{P} }
\def\k{\kappa}
\def\o{\omega}
\def\po{P_\omega}
\def\eo{E_\o}

\allowdisplaybreaks

\begin{document}

\title[Collisions of transient RWRE]
{Arbitrary many walkers meet infinitely often in a subballistic random environment}
\date{\today}

\author{Alexis Devulder}
\address{
Laboratoire de Math\'ematiques de Versailles, UVSQ, CNRS, Universit\'e Paris-Saclay,
78035 Versailles, France.
}
\email{devulder@math.uvsq.fr }

\author{Nina Gantert}
\address{Technische Universit\"at M\"unchen,
Fakult\"at f\"ur Mathematik,
85748 Garching,
Germany }
\email{gantert@ma.tum.de}

\author{Fran\c{c}oise P\`ene}
\address{Univ Brest,  Institut Universitaire de France, UMR CNRS 6205,  Laboratoire de
Math\'ematiques de Bretagne Atlantique, France}
\email{francoise.pene@univ-brest.fr}

\maketitle

\begin{abstract}

We consider $d$ independent walkers in the same random environment in $\Z$.
Our assumption on the law of the environment is such that a single walker is transient to the right but subballistic. We show that - no matter what $d$ is - the $d$ walkers meet infinitely often, i.e. there are almost surely infinitely many times for which all the random walkers are at the same location.
\end{abstract}


\section{Introduction and statement of the main results}

\subsection{Motivation and main result}

Meetings, or collisions, of random walks have been studied since the early works of P{\'o}lya \cite{Polya}.
Indeed, P{\'o}lya \cite{Polya} proves that almost surely, two independent simple random walks meet infinitely often in $\Z$ or $\Z^2$,
but only a finite number of times in $\Z^3$. Proving this result was his main motivation for proving his celebrated result on transience and recurrence of simple random walks in $\Z^d$, $d\geq 1$
(see P{\'o}lya \cite{PolyaInc} "Two incidents").
Moreover, Dvoretzky and Erd{\"o}s \cite{Dvoretzky_Erdos} stated that
almost surely, in $\Z$, $3$ independent simple random walks meet (simultaneously) infinitely often whereas $4$ independent simple random walks meet only a finite number of times.

More recently, the question whether two or three independent random walks meet infinitely often or only a finite number of times has been studied on
some other graphs, for example wedge combs, percolation clusters of $\Z^2$ and some trees.
See e.g.
Krishnapur and Peres \cite{  KrishnapurPeres},
Chen, Wei and Zhang \cite{Chen_Wei_Zhang_08},
Chen and Chen \cite{Chen_Chen_10}, \cite{Chen_Chen_11},
Barlow, Peres and Sousi \cite{BarlowPeresSousi}, Hutchcroft and Peres \cite{HP15} and Chen \cite{Chen16}.
For some applications of collisions of random walks in physics and in biology, we refer to
Campari and Cassi \cite{CampariCassi}.

Meetings/collisions  have also been considered recently for random walks in random environments (RWRE) in $\Z$.
It was proved by Gantert, Kochler and P\`ene \cite{NMFa} that almost surely, $d$ independent random walks in the same (recurrent) random environment meet infinitely often in the origin for any $d\geq 2$. In other words, the $d$ walkers together are a recurrent Markov chain on
$\Z^d$.
Related questions have been studied by
Gallesco \cite{Gallesco} and Zeitouni \cite[p. 307]{Z01}.
A necessary and sufficient condition for
several independent (recurrent) Sinai walks and simple random walks
to meet infinitely often (even if the walkers together are a transient Markov chain) has been
proved by Devulder, Gantert and P\`ene
\cite{Devulder_Gantert_Pene}.
Systems of several random walks in the same random environment have also been investigated
in some other papers,
see e.g. Andreoletti \cite{Andreoletti_system}, and
Peterson et al. \cite{Peterson_Sepp}, \cite{Peterson_2010} and \cite{Jara_Peterson},
but these papers do not consider collisions, which are our main object of study.

The aim of the present paper is to
investigate if there are infinitely many meetings in the (zero speed) transient case,
that is, when each random walk is transient with zero speed in the same random environment. In this case, the $d$ walkers together are clearly a transient
Markov chain, but it turns out that - no matter what $d$ is - the $d$ walkers meet infinitely often, almost surely.

Let $\o:=(\o_x)_{x\in\Z}$ be a collection of i.i.d random variables,
taking values in $(0,1)$ and with joint law $\p$.
A realization of $\o$ is called an {\it environment}.
Let $d\geq 2$, and $(x_1,\dots, x_d)\in\Z^d$.
Conditionally on $\o$, we consider $d$ independent Markov chains
$\big(S_n^{(j)}\big)_{n\in\N}$, $1\leq j \leq d$,
defined for all $1\leq j \leq d$ by $S_0^{(j)}=x_j$
and for every $n\in\N$, $y\in\Z$ and $z\in\Z$,
\begin{equation}\label{probatransition}
    \po^{(x_1,\dots, x_d)}\big(S_{n+1}^{(j)}=z|S_n^{(j)}=y\big)
=
\left\{
\begin{array}{ll}
\o_y & \text{ if } z=y+1,\\
1-\o_y & \text{ if } z=y-1,\\
0      & \text{ otherwise}.
\end{array}
\right.
\end{equation}
We say that each $\big(S_n^{(j)}\big)_{n\in\N}$ is a {\it random walk in random environment} (RWRE).
So, $S^{(j)}:=\big(S_n^{(j)}\big)_{n\in\N}$, $1\leq j \leq d$ are $d$ independent random walks in the same (random) environment $\o$.
Here and throughout the paper, $\N$ denotes the set of nonnegative integers, including $0$.

The probability $\po^{(x_1,\dots, x_d)}$ is called the {\it quenched law}.
We also consider the {\it annealed law}, which is defined by
$$
    \P^{(x_1,\dots,x_d)}(\cdot)=\int\po^{(x_1,\dots,x_d)}(\cdot)\p(\textnormal{d}\o).
$$
Notice in particular that $\big(S_n^{(j)}\big)_{n\in\N}$, $1\leq j \leq d$, is not Markovian under $\P^{(x_1,\dots,x_d)}$.
We sometimes consider $S_n:=S_n^{(1)}$, $n\in\N$.
We then denote by $\po^x$ the quenched law for a single RWRE starting at $x\in\Z$, and write $\po$ instead of $\po^x$ when $x=0$,
and
$
    \P(\cdot)=\int\po(\cdot)\p(\textnormal{d}\o)
$
the corresponding annealed law.
We denote by $\E$, $\eo^{(x_1,\dots,x_d)}$, $\EE$, $\eo$ and $\eo^x$
the expectations under $\p$, $\po^{(x_1,\dots,x_d)}$, $\P$, $\po$ and $\po^x$ respectively. We also introduce
$$
    \rho_x
:=
    \frac{1-\o_x}{\o_x},
\qquad
    x\in\Z.
$$
We assume that there exists
$\varepsilon_0\in(0,1/2)$ such that
\begin{equation}\label{UE}
    \p\big[\omega_0\in[\varepsilon_0,1-\varepsilon_0]\big]=1.
\end{equation}
This ellipticity condition is standard for RWRE.
Solomon \cite{S75} proved that $(S_n)_n$ is $\p$-almost surely recurrent
when $\E(\log \rho_0)=0$, and  transient to the right
when $\E(\log \rho_0)<0$.
In this paper we only consider the transient, subballistic case.
More precisely, we assume that
\begin{equation}\label{rho<0}
  \E(\log \rho_0)<0
\end{equation}
hence $\big(S_n\big)_{n\in\N}$ is a.s. transient to the right, and
\begin{align}
& \label{ConditionA}
    \text{ there exists $0<\k<1$ such that }
    \E\big(\rho_0^\k\big)=1.
\end{align}
In this case, $\kappa$ is unique and
it has been proved by Kesten, Koslov and Spitzer  \cite{KestenKozlovSpitzer}
that, under the additional hypothesis that $\log\rho_0$ has a non-arithmetic distribution,
$S_n$ is asymptotically of order $n^{\k}$ as $n\to+\infty$ (we do not make this assumption in the present paper).
Notice that due to our hypotheses \eqref{rho<0} and \eqref{ConditionA}, $\omega_0$ is non-constant.
This excludes the degenerate case of simple random walks, which behave very differently.

\medskip

Our main result is the following.

\medskip

\begin{thm}\label{Theorem1Meeting}
Assume \eqref{UE}, \eqref{rho<0} and \eqref{ConditionA}. Let $x_1,...,x_d\in \mathbb Z$ with the same parity, where $d\geq 2$.
Then $\P^{(x_1,\dots,x_d)}$-almost surely, there exist infinitely many $n\in\N$ such that
$$
    S_n^{(1)}= S_n^{(2)}=\ldots=S_n^{(d)}.
$$
\end{thm}
Notice in particular that each random walk $S^{(j)}$ is transient, however $d$ independent random walks $S^{(j)}, 1 \leq j \leq d$ meet simultaneously infinitely often, for every $d>1$. Theorem \ref{Theorem1Meeting}
remains true for $-1<\k<0$ when $\E(\log \rho_0)>0$
by symmetry (i.e., by replacing $\omega_j$ by $1-\omega_{-j}$ for every $j\in\Z$).
To the best of our knowledge, this is the first example of a transient random walk, for which arbitrary many independent realizations of this
random walk meet simultaneously infinitely often.

Also, notice that we do not require a non-arithmetic distribution in our Theorem \ref{Theorem1Meeting}.


In this paper, we use quenched techniques. For more information and results about quenched techniques for RWRE in the case $0<\k<1$,
we refer to Enriquez, Sabot and Zindy \cite{ESZ2}, \cite{ESZ3}, Peterson and Zeitouni \cite{Peterson_Zeitouni}
and Dolgopyat and Goldsheid \cite{Dolgopyat_Goldsheid}. For a quenched study of the continuous time analogue of RWRE, that is, diffusions in a drifted Brownian potential expressed in terms of $h$-extrema, we refer to Andreoletti et al. \cite{AndreolettiDevulder} and \cite{AndreolettiDevulderVechambre}.

Our results may be rephrased in terms of the capacity of the diagonal $\{y\in\Z^d: y_1 = \ldots =y_d\}$ for the Markov chain
$\big(S_n^{(1)}, S_n^{(2)}, \dots, S_n^{(d)}\big)$ on $\Z^d$, see \cite{BePePe}.
\subsection{Sketch of the proof and organization of the paper}
We introduce, in Section \ref{Sect_Potentiel}, the potential $V$ for the RWRE, and we provide some useful estimates concerning this potential.
In particular, we prove a key lemma, Lemma \ref{LemmaConstantInvariantMeasure}, which will enable us to
control the invariant probability measure of a RWRE reflected inside a typical valley.

In Section \ref{Sect_Valleys}, we show that $\p$-almost surely, there exist  successive valleys $[a_i, c_i]$, $i\geq 1$,  for the potential $V$,
with bottom $b_i$ and depth at least $f_i\approx \log (i!)$ for some deterministic sequence $(f_i)_i$,
with $a_i<b_i<c_i<a_{i+1}<\dots$, so that the valleys $[a_i, c_i]$ are disjoint.

We also show (see Propositions \ref{BorelCantelli2} and \ref{Prop_i_n}) that almost surely for infinitely many $i$,
{\bf (a)} the depth of the $i$-th valley $[a_i,c_i]$ is larger than $f_i+z_i$
for some deterministic $z_i$,
and there exists no valley of depth at least $f_i$ before it,
and {\bf (b)} the invariant probability measure of a RWRE reflected inside this valley is uniformly large at the bottom $b_i$ of the $i$-th valley (due to Lemma \ref{LemmaConstantInvariantMeasure}).
We denote the sequence of these indices $i$ by $(i(n))_{n\in\N}$. The valleys number $i(n)$, $n\in\N$,  are called the {\it very deep valleys}.
See Figure \ref{figure2_Potential_Valley_Transient} for the pattern of the potential for a very deep valley.

The advantage of such a very deep valley is that thanks to ${\bf (a)}$, with high probability, the particles $S^{(1)}, \dots, S^{(d)}$ will arrive quickly to the bottom $b_{i(n)}$ of this valley (see Lemma \ref{LemmaHittingTimeVeryDeepValleys}), and stay together a long time inside this valley (see Lemma \ref{LemmaStayingInside}).
This will ensure, thanks to a coupling argument (detailed in Subsection \ref{SubSecEndofSketch})
and to ${\bf (b)}$, that $S^{(1)},...,S^{(d)}$ will meet simultaneously
in the valley of bottom $b_{i(n)}$ with a quenched probability greater than some strictly positive constant.
We conclude in Subsection \ref{Sect_conclusion}
that
with a quenched probability greater than some strictly positive constant,
$S^{(1)},...,S^{(d)}$ will meet simultaneously for an infinite number of $n$.
This, combined with a result of Doob, will ensure that $\P^{(x_1,\dots,x_d)}$-almost surely,
$S^{(1)},...,S^{(d)}$ will meet simultaneously for an infinite number of $n$.

Finally, we prove in Section \ref{Sect_RW_conditionned} two identities in law related to random walks conditioned to stay nonnegative
or conditioned to stay strictly positive.


\section{Potential and some useful estimates}\label{Sect_Potentiel}

\subsection{Potential and hitting times}
We recall that the {\it potential} $V$ is a function of the environment $\o$, which is defined on $\Z$ as follows:
\begin{equation}\label{eq_def_V}
    V(x)
:=\left\{
    \begin{array}{lr}
    \sum_{k=1}^x \log\frac{1-\o_k}{\o_k}
    &
    \textnormal{if } x>0,\\
    0 & \textnormal{if }x=0,\\
    -\sum_{k=x+1}^0 \log\frac{1-\o_k}{\o_k} &  \textnormal{if } x<0.
    \end{array}
    \right.
\end{equation}
Here and throughout the paper, $\log$ denotes the natural logarithm.
We will write
\begin{equation}\label{defVminus}
  V(-\cdot) \textnormal{ for the process } (V(-y),\ y \in \Z)\, .
  \end{equation}
 For $x\in\Z$,
we define the hitting times of $x$ by $(S_n)_n$
by
\begin{equation}\label{tautauj}
    \tau(x)
:=
    \inf\{k\in\N\ :\ S_k=x\},
\end{equation}
where $\inf\emptyset=+\infty$ by convention.
We also define for
$x\geq 0$ and $y\in\Z$,
\begin{equation}\label{tautaujxy}
    \tau(x,y)
:=
    \inf\{k\in \N\ :\ S_{\tau(x)+k}=y\}.
\end{equation}
We recall the following facts, which explain why the potential plays a crucial role
for RWRE.

\begin{fact}[Golosov \cite{Golosov84}, Lemma 7 and Shi and Zindy \cite{ShiZindy} eq. (2.4) and (2.5) with a slight modification for the second inequality, obtained from the first one by symmetry]
\label{FactInegProbaAtteinte}
We have,
\begin{eqnarray}
\label{InegProba1}
    P_\omega^b[\tau(c)<k]
& \leq & k \exp\left(\min_{\ell \in [b,c-1]}V( \ell)-V(c-1)\right),\qquad     b<c\, ,
\\
\label{InegProba2}
    P_\omega^b[\tau(a)<k]
& \leq & k \exp\left(\min_{\ell \in [a,b-1]}V(\ell) -V(a)\right),\qquad     a<b\, .
\end{eqnarray}
\end{fact}
The following statements about hitting times can be found e.g. respectively in
\cite[(2.1.4)]{Z01} and \cite[Lemma 2.2]{Devulder_Persistence} coming from \cite[p. 250]{Z01}.
\begin{fact}
Let $a<b<c$ be three integers.Then
\begin{eqnarray}
\label{EgProba4}
    P_\omega^b[\tau(c)<\tau(a)]
& = &
    \left(\sum_{k=a}^{b-1}e^{V(k)}\right)\left(\sum_{k=a}^{c-1}e^{V(k)}\right)^{-1}\, ,
\\
\label{InegProba4}
    E_\omega^b[\tau(a)\wedge\tau(c)]
& \leq & \varepsilon_0^{-1}(c-a)^2 \exp\left(\max_{a\le\ell\le k\le c-1,\, k\ge b}(V(k)-V(\ell))\right)\, ,
\\
\label{InegProba3}
    E_\omega^b[\tau(a)\wedge\tau(c)]
& \leq & \varepsilon_0^{-1}(c-a)^2 \exp\left(\max_{a\le\ell\le k\le c-1,\, \ell\le b-1}(V(\ell)-V(k))\right)\, ,
\end{eqnarray}
where we write $a\wedge b=\min(a,b)$. Note that \eqref{InegProba4} comes from a comparison with the RWRE reflected at $a$ and  \eqref{InegProba3} comes from a comparison with the RWRE reflected at $c$.
\end{fact}

\subsection{Excursions of the potential}
We now define by induction, as Enriquez, Sabot and Zindy in  \cite{ESZ2} and \cite{ESZ3},
the weak descending ladder epochs for $V$ as
\begin{equation}\label{eq_def_ei}
    e_0
:=
    0,
\qquad
    e_{i}
:=
    \inf\{k>e_{i-1}\ :\ V(k)\leq V(e_{i-1})\},
\qquad
    i\geq 1.
\end{equation}
In particular, the excursions $(V(k+e_i)-V(e_i),\ 0\leq k\leq e_{i+1}-e_i)$, $i\geq 0$ are i.i.d.

Recall that, by classical large deviation results (see e.g. \cite{ESZ2} Lemma 4.2, stated when $\log \rho_0$
has a non-lattice distribution but also valid in the lattice case),
there exists a function $I$ such that $I(0)>0$
and
\begin{equation}\label{eq_Large_Dev_RW}
    \forall k\geq 1, \forall y\geq 0,
\qquad
    \p[V(k)\geq k y]
\leq
    \exp[-k I(y)].
\end{equation}
One consequence of such large deviations is that $\E(e_1^2)<\infty$, since for every $k\geq 1$,
$\p(e_1> k)\leq \p[V(k)\geq 0]\leq \exp[-k I(0)]$.

We also introduce the height $H_i$ of the excursion $[e_i,e_{i+1}]$ of the potential, that is,
\begin{equation}\label{eq_def_Hi}
    H_i
:=
    \max_{e_i\leq k\leq e_{i+1}}[V(k)-V(e_i)],
\qquad
    i\geq 0.
\end{equation}

Recall that due to a result of Cram\'er (see e.g. \cite{Korshunov}, Theorem 3 in the lattice and in the non-lattice case, see also \cite[Prop 2.1]{Peterson}), there exists $C_M>0$ such that
\begin{equation}\label{eq_sup_RW}
    \p\left(\sup_{y\geq 0} V(y) \geq z\right)
\sim_{z\to+\infty, z\in \Gamma}
    C_M e^{-\k z},
\end{equation}
where $\Gamma=\mathbb R$ if $V$ is non-lattice and where $\Gamma=a\mathbb Z$ if the potential $V$ is lattice
and if $a\in\mathbb R_+^*$ is the largest positive real number such that $\p(V(n)\in a\mathbb Z,\ \forall n\ge 1)=1$.
The same result is true with $>$ instead of $\geq$, with $C_M$ replaced by a positive constant $C_M'$
($C_M'=C_M$ is the non-lattice case but
$C_M'=C_M e^{-\k a}$ in the lattice case with the previous notation), i.e. we have
\begin{equation}\label{eq_sup_RWprime}
    \p\left(\sup_{y\geq 0} V(y) > z\right)
\sim_{z\to+\infty, z\in \Gamma}
    C_M' e^{-\k z}.
\end{equation}
As a consequence, there exists a constant $C_M''\geq 1$ such that
\begin{equation}\label{eq_sup_RW_2}
    \forall z \in\R,
\qquad
    \p\left(\sup_{y\geq 0} V(y)\geq z\right)
\leq
    C_M'' e^{-\k z}.
\end{equation}

We will use in the rest of the paper the following result.
It is already known when the distribution of $\log\rho_0$ is non-lattice and we extend it for the lattice case.

\begin{prop}\label{FactIglehart}
(adapted from Iglehart \cite{Iglehart})
Under Hypotheses \eqref{UE}, \eqref{rho<0} and \eqref{ConditionA}, there exists $C_I>0$ such that
\begin{equation}\label{eq:Iglehart}
    P(H_1> h)\sim_{h\to+\infty, h\in \Gamma} C_I e^{-\k h}\quad\mbox{as}\quad h\rightarrow\infty\ \mbox{ in }\ \Gamma\, ,
\end{equation}
where $\Gamma=\mathbb R$ if $V$ is non-lattice and where $\Gamma=a\mathbb Z$ if the potential $V$ is lattice and if $a\in\mathbb R_+^*$ is the largest positive real number such that $\p(V(n)\in a\mathbb Z,\ \forall n\ge 1)=1$.
\end{prop}
\begin{proof}[Proof]
When $V$ is non-lattice, the result is proved in \cite[Thm 1]{Iglehart}. When $V$ is lattice, the proof is the same as the proof of
\cite[Thm 1]{Iglehart} with the use of \eqref{eq_sup_RWprime}
instead of \cite[Lemma 1]{Iglehart}. However, since the proof is short, we write it below with our notations.

First, notice that for $z>0$,
\begin{eqnarray*}
    \p\Big[\sup_\N V>z\Big]
& = &
    \p\Big[\sup_\N V>z, H_0>z\Big]
+
    \p\Big[\sup_\N V>z, H_0\leq z\Big]
\\
& = &
    \p\Big[H_0>z\Big]
+
    \E\bigg[\un_{\{H_0 \leq z\}}\p\bigg(\sup_{\N} V >z-y\bigg)\bigg|_{y=V(e_1)}\bigg]
\end{eqnarray*}
by the strong Markov property applied at stopping time $e_1$. Hence,
\begin{equation}\label{Eq_Egalite_Iglehart}
    e^{\k z} \p\Big[H_0>z\Big]
=
    e^{\k z}\p\Big[\sup_\N V>z\Big]
-
    \E\bigg[\un_{\{H_0 \leq z\}}e^{\k V(e_1)}\alpha(z)\bigg],
\end{equation}
where
$$
    \alpha(z)
:=
    \bigg(e^{\k(z-y)}\p\bigg(\sup_{\N} V >z-y\bigg)\bigg)\bigg|_{y=V(e_1)}.
$$
Notice that when $z\to +\infty$ with $z\in\Gamma$,
$\un_{\{H_0 \leq z\}}\to 1$ a.s. since $V$ is transient to $-\infty$, whereas
$\alpha(z)\to C_M'$ a.s. by \eqref{eq_sup_RWprime} since $(z-V(e_1))\in \Gamma$ and goes to infinity.
Also,
$\sup_{z>0}\big|\un_{\{H_0 \leq z\}}e^{\k V(e_1)}\alpha(z)\big|\leq C_M''$ a.s.
by \eqref{eq_sup_RW_2} and since $V(e_1)\leq 0$. So by the dominated convergence theorem,
$$
    \E\bigg[\un_{\{H_0 \leq z\}}e^{\k V(e_1)}\alpha(z)\bigg]
\to_{z\to+\infty, z\in\Gamma}
    C_M'\E\big[e^{\k V(e_1)}\big].
$$
This together with \eqref{Eq_Egalite_Iglehart} and \eqref{eq_sup_RWprime} gives
$    e^{\k z} \p\Big[H_0>z\Big]
\to_{z\to+\infty, z\in\Gamma}
    C_M'\big(1-\E\big[e^{\k V(e_1)}\big]\big)
=: C_I>0
$,
which ends the proof in the lattice and in the non-lattice case.
\end{proof}
Observe that, in the lattice case, \eqref{eq:Iglehart} is not true for $h\rightarrow\infty$ in $\mathbb R$ (it is even not true in $\frac 12\Gamma$) but that, in any case, under Hypotheses \eqref{UE}, \eqref{rho<0} and \eqref{ConditionA}, there exist $C_I^{(0)}>0$ and $C_I^{(1)}>0$ such that
\begin{equation}\label{Igelhart0}
    \forall h\in\mathbb R_+,
\quad
    C_I^{(0)} e^{-\k h}\le P(H_1> h)\le  P(H_1\ge h)\le C_I^{(1)} e^{-\k h}\, .
\end{equation}

\subsection{An estimate for the hitting time of the first valley}
Let us define the maximal increase of the potential between $0$ and $x$, and then between $x$ and $y$, by
\begin{eqnarray*}
    V^{\uparrow}(x)
& := &
    \max_{0\leq i \leq j \leq x}[V(j)-V(i)],
\qquad
    x\in\N,
\\
    V^{\uparrow}(x,y)
& := &
    \max_{x\leq i \leq j \leq y}[V(j)-V(i)],
\qquad
    x\leq y.
\end{eqnarray*}
We also introduce
\begin{eqnarray}
    T^{\uparrow}(h)
& := &
    \min\{x\geq 0,\ V^{\uparrow}(x)\geq h\},
\qquad
    h>0,
\nonumber\\
    m_1(h)
& := &
    \min\Big\{x\geq 0,\ V(x)=\min_{[0, T^{\uparrow}(h)]} V\Big\},
\qquad
    h>0,
\label{eq_def_m1}
\\
    {}^{\uparrow}T_h(y)
& := &
    \max\{x\in\Z,\, x\leq m_1(h), \ V(x)-V[m_1(h)] \geq y\},\, \quad\ h>0,\, y>0,~~~~
\label{eq_def_Tfleche_Gauche}
\\
    T_V(A)
& := &
    \min\{x\geq 1,\ V(x)\in A\}, \qquad A\subset \R,
\label{eq_Def_TV}
\\
    T_{V(-\cdot)}(A)
& := & 
    \min\{x\geq 1,\ V(-x)\in A\}, \qquad A\subset \R,
\label{eq_Def_TV-}
\end{eqnarray}
where we take the convention $\max\emptyset = -\infty$ in \eqref{eq_def_Tfleche_Gauche}.
\smallskip

The following inequality (the proof of which uses our assumptions and Proposition \ref{FactIglehart}) is useful to bound the expectation of the hitting time of the first valley of height $h$ for $(V(x), x\geq 0)$.
It is proved in \cite{ESZ2} when the distribution of $\log \rho_0$ is non-lattice and we prove that it remains true in the lattice case.

\begin{fact}(adapted from \cite[Lemma 4.9]{ESZ2})\label{FactTimeSpnetBeforeFirstValley}
Under hypotheses \eqref{UE}, \eqref{rho<0} and \eqref{ConditionA},
there exists $C_0>1$ such that
\begin{equation}\label{eq_Lemma_SEZ}
    \forall h>0,
\qquad
    \EE_{|0}\big[\tau\big(T^{\uparrow}(h)-1\big)\big]
\leq
    C_0 e^h,
\end{equation}
where $\EE_{|0}$ denotes the expectation under the annealed law
$\PP_{|0}$ of
the RWRE $(S_n)_n$ on $\mathbb N$ starting from $0$ and reflected at $0$
(that is, on the environment $\omega$ (under $\p$) with $\omega_0$ replaced by $1$).
\end{fact}

\begin{proof}
Assume \eqref{UE}, \eqref{rho<0} and \eqref{ConditionA}. If the distribution of $\log \rho_0$ is non-lattice,
\eqref{eq_Lemma_SEZ} is proved in \cite[Lemma 4.9]{ESZ2} for large $h$ and then for all $h>0$ up to a change of
the value of $C_0$, which we can choose to be greater than $1$.

Notice that at least up to their equation (4.15),
the proof of \cite[Lemma 4.9]{ESZ2} does not use
$P(H_1> h)\sim_{h\to+\infty} C_I e^{-\k h}$
but only uses \eqref{Igelhart0} instead of it.
Since \eqref{Igelhart0} is true when the distribution of $\log \rho_0$ is lattice,
the proof of \cite[Lemma 4.9]{ESZ2}, which does not use the non-lattice hypothesis anywhere else,
remains valid in the lattice case at least up to their equation (4.15).

So it only remains to prove that $\E[J_0 |H_0<h]\leq C e^{(1-\k)h}$ for some $C>0$ for large $h$,
where $J_0:=\sum_{k=0}^{e_1} e^{V(k)}$.
We simplify the rest of the proof as follows.
For $x\in\R$, we denote by $\lfloor x\rfloor$ the integer part of $x$.
Since $\p[H_0<h]\to 1$ when $h\to+\infty$
and $0\leq V(k)\leq h$ for $0\leq k \leq e_1$ on $\{H_0<h\}$, we have for large $h$,
\begin{eqnarray*}
    \E\big[J_0 |H_0<h\big]
& \leq &
    \E\bigg[\bigg(\sum_{k=0}^\infty e^{V(k)}\un_{\{0\leq V(k)\leq h\}}\bigg)\frac{\un_{\{H_0<h\}}}{\p[H_0<h]}\bigg]
\\
& \leq &
    2\E\bigg[\sum_{k=0}^\infty e^{V(k)}\sum_{p=0}^{\lfloor h\rfloor}\un_{\{p\leq V(k)<p+1\}}\un_{\{H_0<h\}}\bigg]
\\
& \leq &
    2\E\bigg[\sum_{k=0}^\infty \sum_{p=0}^{\lfloor h\rfloor} e^{p+1}\un_{\{p\leq V(k)<p+1\}}\bigg]
\\
& = &
    2\sum_{p=0}^{\lfloor h\rfloor} e^{p+1}\E\bigg[\un_{\{T_V([p,p+1[)<\infty\}}\sum_{k=0}^\infty \un_{\{p\leq V(k)<p+1\}}\bigg]
\\
& = &
    2\sum_{p=0}^{\lfloor h\rfloor} e^{p+1}
    \E\bigg[\un_{\{T_V([p,p+1[)<\infty\}}\E\bigg(\sum_{k=0}^\infty \un_{\{p\leq V(k)+x<p+1\}}\bigg)\bigg|_{x=V\left(T_V([p,p+1[)\right)}\bigg]
\end{eqnarray*}
by the strong Markov property. Hence for large $h$,  
\begin{equation*}
    \E\big[J_0 |H_0<h\big]
\leq
    2\sum_{p=0}^{\lfloor h\rfloor} e^{p+1}
    \p\Big[T_V([p,p+1[)<\infty\Big]
    \E\bigg(\sum_{k=0}^\infty \un_{\{-1\leq V(k)\leq 1\}}\bigg).
\end{equation*}
But the expectation in the right hand side of the previous line is finite because $V$ is a random walk transient to $-\infty$,
whereas
$
    \p\big[T_V([p,p+1[)<\infty\big]
\leq
    \p\big[\sup_{\N} V\geq p\big]
\leq
    C_M'' e^{-\k p}
$
by \eqref{eq_sup_RW_2}, so there exists $C>0$ such that for large $h$,
\begin{equation*}
    \E\big[J_0 |H_0<h\big]
\leq
    C\sum_{p=0}^{\lfloor h\rfloor} e^{(1-\k)p}
=
    O\big(e^{(1-\k)h}\big)
\end{equation*}
as $h\to+\infty$ since $0<\k<1$.
This ends the proof of \eqref{eq_Lemma_SEZ} also when $\log \rho_0$ has a lattice distribution, for large $h$ and then for all $h>0$ up to a change of constant.
So Fact \ref{FactTimeSpnetBeforeFirstValley} is proved in both cases.
\end{proof}

\subsection{An estimate for the invariant measure}
The following lemma will be useful to control the invariant probability measure of a RWRE reflected inside our valleys,
and to show that the invariant measure at the bottom of some of these valleys (introduced below in Proposition \ref{Prop_i_n})
is uniformly large.

\begin{lem}\label{LemmaConstantInvariantMeasure}
Under hypotheses \eqref{UE}, \eqref{rho<0} and \eqref{ConditionA},
there exist constants $C_{2}>0$ and $h_0>0$ such that for every $h>h_0$,
\begin{equation}\label{eqLemmaConstantInvariantMeasure1}
    \E\Bigg[\sum_{m_1(h)\leq x \leq T^{\uparrow}(h)} e^{-(V(x)-V(m_1(h))} \Bigg]
\leq
    C_{2}
\end{equation}
and
\begin{equation}\label{eqLemmaConstantInvariantMeasure2}
    \E\Bigg[\sum_{{}^{\uparrow}T_h(h)\leq x<m_1(h)} e^{-(V(x)-V(m_1(h))}
\ \bigg |\,  {}^{\uparrow}T_h(h)\ge 0\Bigg]
\leq
    C_{2}.
\end{equation}
\end{lem}
\begin{proof}
Recall that $\omega_0$ is non degenerate, which together with \eqref{ConditionA} gives
$\p[\log \rho_0>0]>0$ and $\p[\log \rho_0<0]>0$. Also $\lim_{x\to+\infty}V(x)=-\infty$ by \eqref{rho<0},
so we can use Propositions \ref{Lemma_Egalite_Loi_Gauche_de_m1} and \ref{LemmaTanakaStyleForRW}.
Since $(V(x+m_1(h))-V(m_1(h)), \ 0\leq x \leq T^{\uparrow}(h)-m_1(h))$
is equal in law to
$(V(x),\, 0\leq x\leq T_V([h,+\infty[)\, |\, T_V([h,+\infty[)<T_V(]-\infty,0[))$ by Proposition \ref{LemmaTanakaStyleForRW},
we have,
\begin{eqnarray}
&&
    \E\bigg(\sum_{m_1(h)\leq x \leq T^{\uparrow}(h)} e^{-(V(x)-V(m_1(h))} \bigg)
\nonumber\\
& = &
    \E\bigg(\sum_{0\leq x \leq T_V([h,+\infty[)} e^{-V(x)} \Big|T_V([h,+\infty[)<T_V(]-\infty,0[)\bigg)
\nonumber\\
& = &
    \sum_{x=0}^\infty \E\left(\un_{\{x \leq T_V([h,+\infty[)\}} e^{-V(x)} \Big|T_V([h,+\infty[)<T_V(]-\infty,0[)\right).
\label{Ineg_1_Proof3}
\end{eqnarray}
Note that $(V(x))_{x\in\N}$ is a Markov chain starting from $0$. We will write $\p^y$ for the law of this Markov chain starting from $y\in\R$, which is the law of  $(y+V(x))_{x\in\N}$ (we will use this notation in \eqref{Ineg_Proof3_b} below with $y= V(x)$).
We first notice that for all $x>0$ and $h>0$,
\begin{equation}\label{Ineg_Proof3_a}
    \E\left(\un_{\{x \leq T_V([h,+\infty[)\}} e^{-V(x)} \un_{\{V(x)\geq \sqrt{x}\}}
    \Big|T_V([h,+\infty[)<T_V(]-\infty,0[)\right)
\leq
    e^{-\sqrt{x}}.
\end{equation}
Let $h_0:= 8\k /I(0)$, where $I(0)>0$ is introduced in \eqref{eq_Large_Dev_RW}.
Hence, using $e^{-V(x)}\leq 1$ and
$
    \p\big[T_V([h,+\infty[)<T_V(]-\infty,0[)\big]
\geq
    \p\big[T_V([h,+\infty[)<T_V(]-\infty,0])\big]
\geq
    C_I^{(0)} e^{-\k h}
$
(see \eqref{Igelhart0}) in the first inequality,
then, due to the Markov property at $x$,
we have for all $h>0$ and $x\geq 0$,
\begin{eqnarray}
&&
    \E\left(\un_{\{x \leq T_V([h,+\infty[)\}} e^{-V(x)} \un_{\{V(x)< \sqrt{x}\}}
    \Big|T_V([h,+\infty[)<T_V(]-\infty,0[)\right)
\nonumber\\
& \leq &
    [C_I^{(0)}]^{-1}e^{\k h}\E\left(\un_{\{x \leq T_V([h,+\infty[)\}} \un_{\{V(x)< \sqrt{x}\}}
    \un_{\{T_V([h,+\infty[)<T_V(]-\infty,0[)\}}\right)
\nonumber\\
& = &
    \frac{e^{\k h}}{C_I^{(0)}}\E\left(\un_{\{x \leq T_V([h,+\infty[), \ V(x)< \sqrt{x}, \ x<T_V(]-\infty,0[)\}}
    \p^{V(x)}[T_V([h,+\infty[)<T_V(]-\infty,0[)]\right)
\nonumber\\
& \leq &
    [C_I^{(0)}]^{-1}e^{\k h}\E\left(\un_{\{V(x)\geq 0\}}
        \p\left[\sup_{y\geq 0}V(y)\geq h-\sqrt{x}\right]\right)
\nonumber\\
& \leq &
    [C_I^{(0)}]^{-1}e^{\k h}\p[V(x)\geq 0] C_M'' e^{-\k (h-\sqrt{x})}
\nonumber\\
& \leq &
    [C_I^{(0)}]^{-1}e^{-xI(0)} C_M'' e^{\k \sqrt{x}},
\label{Ineg_Proof3_b}
\end{eqnarray}
where we used \eqref{eq_sup_RW_2} and \eqref{eq_Large_Dev_RW} in the last two inequalities.

Finally, using \eqref{Ineg_1_Proof3} and summing
\eqref{Ineg_Proof3_a} and  \eqref{Ineg_Proof3_b}
over $x$, we get for every $h>h_0$,
\begin{equation}
    \E\bigg(\sum_{m_1(h)\leq x \leq T^{\uparrow}(h)} e^{-(V(x)-V(m_1(h))} \bigg)
\leq
    \sum_{x\geq 0} \bigg(e^{-\sqrt{x}}
    +
    \frac{C_M ''}{C_I^{(0)}}e^{\k \sqrt{x}-xI(0)}\bigg)
=:
    C_{2}<\infty
\end{equation}
since $I(0)>0$, which proves \eqref{eqLemmaConstantInvariantMeasure1} for all $h>h_0$.

We now turn to \eqref{eqLemmaConstantInvariantMeasure2}.
Recall $V(-\cdot)$ and $T_{V(-\cdot)}$ from \eqref{defVminus} and \eqref{eq_Def_TV-}.
Due to Proposition \ref{Lemma_Egalite_Loi_Gauche_de_m1},
the left hand side of \eqref{eqLemmaConstantInvariantMeasure2} is equal to
\begin{eqnarray*}
&&
    \E\bigg[\sum_{0<x\leq T_{V(-\cdot)}([h,+\infty[)}e^{-V(-x)}
    \Big| T_{V(-\cdot)}([h,+\infty[)<T_{V(-\cdot)}(]-\infty,0])\bigg]
\\
& \leq &
    \E\bigg[\sum_{x\geq 0}e^{-V(-x)}\un_{\{V(-x)\geq 0\}}
    \frac{\un_{\{T_{V(-\cdot)}([h,+\infty[)<T_{V(-\cdot)}(]-\infty,0])\}}}{\p[T_{V(-\cdot)}([h,+\infty[)<T_{V(-\cdot)}(]-\infty,0])]}
    \bigg]
\\
& \leq &
    \E\bigg[\sum_{x\geq 0}\sum_{p=0}^{\infty}e^{-V(-x)}\un_{\{p\leq V(-x) <p+1\}}
    \bigg]
    \frac{1}{\p[\inf_{\N^*}V(-\cdot)> 0]}
\\
& \leq &
    \sum_{p=0}^{\infty}e^{-p}
    \E\bigg[\un_{\{T_{V(-\cdot)}([p,p+1[)<\infty\}}\sum_{x\geq 0}\un_{\{p\leq V(-x) <p+1\}}
    \bigg]
    \frac{1}{\p[\inf_{\N^*}V(-\cdot)> 0]}
\\
& \leq &
    \sum_{p=0}^{\infty}e^{-p}
    \E\bigg(\sum_{x\geq 0}\un_{\{-1\leq V(-x)\leq 1\}}\bigg)
    \frac{1}{\p[\inf_{\N^*}V(-\cdot)> 0]}
=:C_3<\infty,
\end{eqnarray*}
due to the strong Markov property in the last line, and since $V(-\cdot)$ is a random walk transient to $+\infty$,
and so the expectation in the previous line is finite
and $\p[\inf_{\N^*}V(-\cdot)> 0]=\p[\sup_{\N^*}V< 0]>0$.
This proves \eqref{eqLemmaConstantInvariantMeasure2} and then the lemma, up to a change of constants.
\end{proof}

\section{Construction of the very deep valleys}\label{Sect_Valleys}
We fix some $\varepsilon\in]0,\frac{1-\kappa}{2\kappa}[$, which is possible since $0<\k<1$ by \eqref{ConditionA}.
We first build a succession of very deep valleys for the potential, with probability $1$. To this end, we set for $i\geq 1$, recalling $C_0$ from Fact \ref{FactTimeSpnetBeforeFirstValley},
\begin{equation}\label{eq_def_N_i}
    N_i
:=
    \big\lfloor C_0 i^{1+\varepsilon}(i!)^{\frac {1+\varepsilon}\kappa}\big\rfloor\, ,
\end{equation}
\begin{equation}\label{eq_def_f_i}
    f_i
:=
    \log (N_i/C_0)-(1+\varepsilon)\log i
\sim_{i\to+\infty}
    \frac{1+\varepsilon}\kappa\log (i!)
\sim_{i\to+\infty}
    \frac{1+\varepsilon}\kappa i \log i\,  .
\end{equation}
We now define (recall that $H_k$ and $e_k$ were respectively defined in \eqref{eq_def_Hi} and \eqref{eq_def_ei}),
$$
    \sigma(0)
:=
    -1,
\qquad
    \sigma(i)
:=
    \inf\{k > \sigma(i-1)\ :\ H_k\geq f_i\},
\qquad
    i>0.
$$
Notice that $\p$-almost surely, $\sigma(i)<\infty$ for all $i\in\N$, since the $H_k$, $k\in\N$ are i.i.d.
and thanks to \eqref{eq:Iglehart}.
We also introduce for $i\geq 1$ (see Figure \ref{figure2_Potential_Valley_Transient}),
\begin{eqnarray*}
    b_i
& := &
    e_{\sigma(i)},
\\
    a_i
& := &\max\left(e_{\sigma(i-1)+1},
    \sup\{k< b_i\, : \, V(k) \geq V(b_i)+f_i +z_i\}\right),
\\
    \alpha_i
& := &
    \max\left(e_{\sigma(i-1)+1},
    \sup\{k< b_i\, : \, V(k) \geq V(b_i)+f_i/2\}\right),
\\
    \gamma_i
& := &
    \inf\{k> b_i\, :\, V(k) \geq V(b_i)+f_i/2\}
    \leq e_{\sigma(i)+1}-1,
\\
    c_i
& := &\min\left(e_{\sigma(i)+1}-1,
\inf\{k> b_i\, :\, V(k) \geq V(b_i)+f_i +z_i\}\right)
\, ,
\end{eqnarray*}
with $z_i:=(\log i)/\kappa$.
We call $[a_i, c_i]$ the {\it valley} number $i$.
Its {\it bottom} is $b_i$, and we call $H_{\sigma(i)}$ its {\it height}.
We also define for $i\geq 1$,
\begin{eqnarray*}
\beta_i^-
& := &
    \sup\{k< b_i\, : \, V(k) \geq V(b_i)+f_i \},
\\
    \beta_i^+
& := &
    \inf\{k> b_i\, :\, \ V(k) \geq V(b_i)+f_i \}.
\end{eqnarray*}

\begin{figure}[htbp]
\includegraphics[scale=1.05]{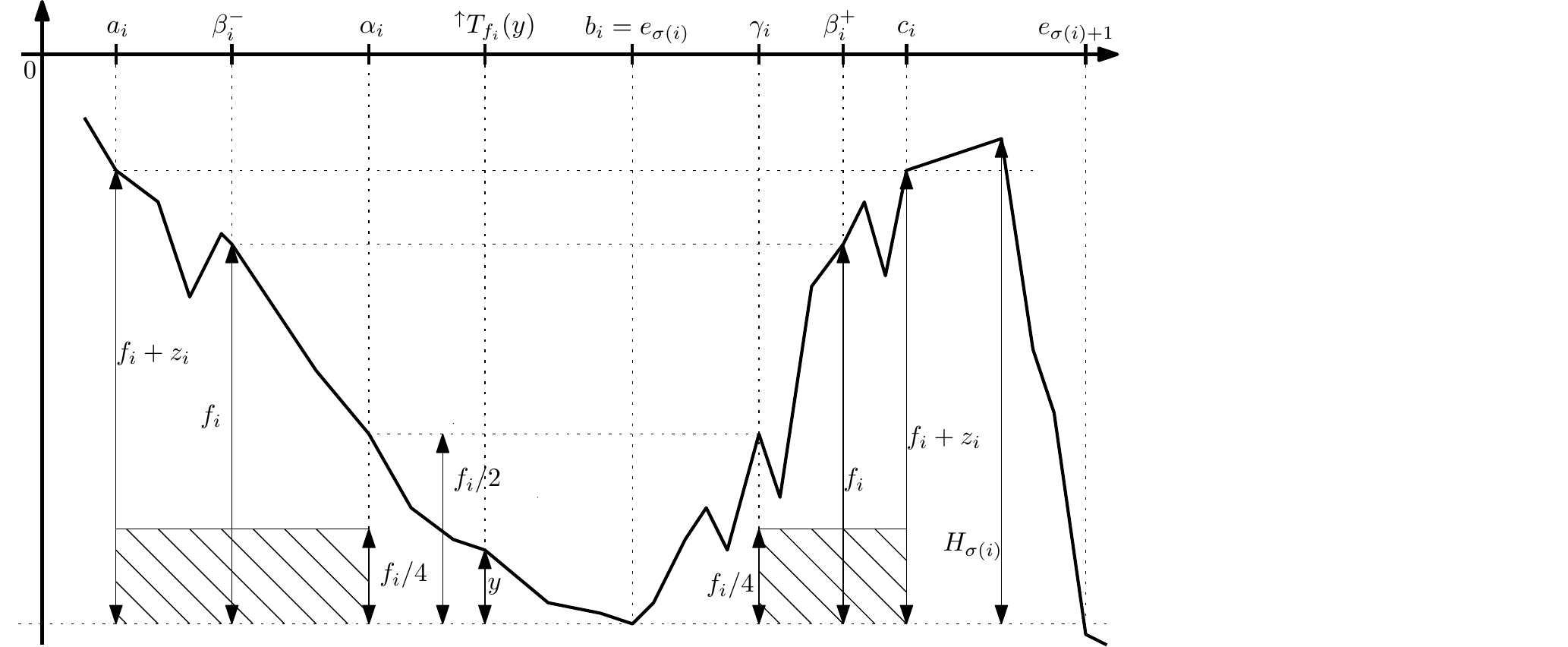}
\caption{Pattern of the potential $V$ for a very deep valley for $i=i(n)$, $\omega\in\widetilde \Omega$ and representation of different quantities.
}
\label{figure2_Potential_Valley_Transient}
\end{figure}

\begin{remark}\label{indep}
Since the r.v. $e_{\sigma(i)+1}$, $i\geq 1$,  are stopping times with respect to the $\sigma$-field
$\big(\sigma(V(1),...,V(n)), \, n\ge 0\big)$, the random variables
$$
    Y_i
:=
    \left(V(z+e_{\sigma(i)})-V(e_{\sigma(i)});
    \ z=e_{\sigma(i-1)+1}-e_{\sigma(i)},...,e_{\sigma(i)+1}-e_{\sigma(i)}\right),
\qquad
    i\geq 1,
$$
are mutually independent by the strong Markov property.
\end{remark}
\begin{rmk}\label{expo<1}
Since the derivative of $x\mapsto\E[e^{x V(1)}]$ at $0$ is $\E[V(1)]<0$
thanks to \eqref{UE}, \eqref{rho<0} and \eqref{ConditionA},
there exists $\kappa_0\in]0,\kappa[$ such that $v_0:=\E[\rho_0^{\kappa_0}]=\E\big[e^{\kappa_0 V(1)}\big]<1$.
\end{rmk}
Let $C_4>\frac{2(\kappa+\k_0)}{|\log v_0|}$.
For every $i\ge 1$, we consider the following sets
\begin{eqnarray*}
    \Omega_1^{(i)}
& := &
    \big\{a_i=\sup\{k< b_i\, :\, V(k) \geq V(b_i)+f_i +z_i\}\big\},
\\
    \Omega_2^{(i)}
& := &
    \left\{b_i\le T^{\uparrow}(f_i)-1\le i\, e^{\kappa f_i}\right\},
\\
    \Omega_3^{(i)}
& := &
    \left\{c_i-a_i\le C_4(f_i+z_i)\right\},
\\
    \Omega_4^{(i)}
& := &
    \left\{\inf_{]a_i,c_i[\setminus ]\alpha_i,\gamma_i[}V> V(b_i)+\frac{f_i}4\right\},
\\
    \Omega_5^{(i)}
& := &
    \left\{\sum_{k=\alpha_i}^{\gamma_i}e^{-(V(k)-V(b_i))}<7C_{2}\right\},
\\
    \Omega_6^{(i)}
& := &
    \left\{H_{\sigma(i)}\ge f_i+z_i\right\} =\left\{c_i=\inf\{k> b_i\, :\, V(k) \geq V(b_i)+f_i +z_i\}\right\}.
\end{eqnarray*}
In particular, the set $\Omega_1^{(i)}$ ensures that the valley $[a_i, c_i]$ of bottom $b_i$ is well separated
from the previous valleys $[a_j, c_j]$, $j<i$,
which we will use for example in Lemma \ref{LemmaStayingInside}.
Also, on $\Omega_2^{(i)}$, the heights of the previous valleys are less than $f_i$,
whereas on $\Omega_6^{(i)}$, the height of the valley $[a_i, c_i]$ is greater than or equal to $f_i+z_i$.
So the valley $[a_i, c_i]$ is significantly higher than the previous ones on $\Omega_2^{(i)}\cap \Omega_6^{(i)}$,
which will be used to get \eqref{Premieremontee}, itself necessary to prove Lemma \ref{LemmaHittingTimeVeryDeepValleys}.
Also, $\Omega_3^{(i)}$ guarantees that the length $c_i-a_i$ of the valley $[a_i, c_i]$ is not too big.
Finally, $\Omega_3^{(i)}$, $\Omega_4^{(i)}$ and $\Omega_5^{(i)}$ help to ensure that the invariant probability measure of a RWRE reflected in the valley $[a_i, c_i]$
is greater than a positive constant at location $b_i$ (which is necessary in Lemma \ref{MINO1}),
and $\Omega_4^{(i)}$ will be useful in Lemma \ref{tempscouplage} to ensure that the coupling
in Section \ref{SubSecEndofSketch} happens quickly, which is itself necessary in Lemma \ref{MINO1}.
We now estimate the probability of these events.

\begin{prop} \label{BorelCantelli1}
We have
$$
    \sum_{j=1}^3\sum_{i\ge 1}\p\left[ \left(\Omega_j^{(i)}\right)^c\right]<\infty,
\qquad
    \sum_{i\ge 1}\p\left[ \left(\Omega_4^{(i)}\right)^c \cap \Omega_6^{(i)}\right]<\infty\,
.$$
\end{prop}
\begin{proof}[Proof of Proposition \ref{BorelCantelli1}]
\begin{itemize}
\item
\underline{Control on $\Omega_1^{(i)}$.}
Observe that for $i\geq 1$,
\begin{eqnarray*}
    \p\left[(\Omega_1^{(i)})^c\right]
&=&
    \p\left[\sup\{k<b_i\, :\, V(k) \ge V(b_i)+f_i+z_i\}<e_{\sigma(i-1)+1}\right]\nonumber\\
&=&
    \p\left[\forall k\in\{e_{\sigma(i-1)+1},...,b_i\},\ V(k) < V(b_i)+f_i+z_i\right]\nonumber\\
&\le&
    \p\left[V(b_{i})> V(e_{\sigma(i-1)+1})-f_i-z_i\right]\nonumber\\
&=&
    \p\left[V(m_1(f_i))> -f_i-z_i\right]
\end{eqnarray*}
by the strong Markov property applied at stopping time $e_{\sigma(i-1)+1}$.
Thus,
\begin{eqnarray}
    \p\left[(\Omega_1^{(i)})^c\right]
&\le&
    \p\left[m_1(f_i)< T_V(]-\infty, -f_i-z_i])\right]
\nonumber\\
&\le&
    \p\left[n_i<T_V(]-\infty,-f_i-z_i])\right]
    +\p\left[m_1(f_i)< T_V(]-\infty,-f_i-z_i])\le n_i\right]
\nonumber\\
&\le&
    \p\left[n_i<T_V(]-\infty,-f_i-z_i])\right]+\p\left[m_1(f_i)< n_i\right],
\label{VD1}
\end{eqnarray}
setting $n_i=\left\lceil 2\kappa_0\frac{f_i+z_i}{|\log v_0|}\right\rceil$.
On the first hand, for every $m>0$ and every positive integer $n$,
we have
$$
    \p\left[T_V(]-\infty,-m])> n\right]
\le
    \p[V(n)\ge -m]
\le
    \E\left[e^{\kappa_0 (V(n)+m)}\right]=v_0^ne^{\kappa_0m},
$$
since $\E\big(e^{\k_0V(n)}\big)=\big[\E\big(e^{\k_0V(1)}\big)\big]^n=v_0^n$ because $(\o_x)_{x\in\Z}$ are i.i.d.
This leads to
\begin{equation}\label{VD2}
    \forall i\geq 1,
\qquad
    \p\left[T_V(]-\infty,-f_i-z_i])>n_i\right]
\le
    v_0^{n_i}e^{\kappa_0(f_i+z_i)}
\le
    e^{-\kappa_0(f_i+z_i)}
\leq
    e^{-\kappa_0 f_i}
\end{equation}
since $z_i\geq 0$.

In order to estimate $\p\left[m_1(f_i)< n_i\right]$, we introduce
the strict descending ladder epochs for $V$ as
\begin{equation}
\label{eq_def_ei_Prime}
    e_0'
:=
    0,
\qquad
    e_{i}'
:=
    \inf\{k>e_{i-1}'\ :\ V(k)< V(e_{i-1}')\},
\qquad
    i\geq 1.
\end{equation}
\begin{equation}
\label{eq_def_Hi_Prime}
    H_i'
:=
    \max_{e_i'\leq k\leq e_{i+1}'}[V(k)-V(e_i')],
\qquad
    i\geq 0.
\end{equation}
In particular for $i\geq 1$, there exists $j\in\N$ such that $m_1(f_i)=e_j'$ and $H_j'\geq f_i$.
Hence, for every $i\geq 1$,
\begin{equation}
    \p\left[m_1(f_i)< n_i\right]
\le
    \p\left[m_1(f_i)< e_{n_i}'\right]
\le
    \p\left[\cup_{j=0}^{n_i-1}\{H_j'\ge f_i\}\right]
\le
    n_i \p[H_1'\geq f_i]
\le
    C_M''n_i e^{-\kappa f_i}
\label{VD3}
\end{equation}
since $\p[H_1'\geq f_i]\leq \p[\sup_{\N}V\geq f_i]\leq C_M'' e^{-\k f_i}$
by \eqref{eq_sup_RW_2}.
Combining \eqref{eq_def_f_i}, \eqref{VD1}, \eqref{VD2}, \eqref{VD3} leads to
$\sum_{i\ge 1}\p\left[(\Omega_1^{(i)})^c\right]<\infty$ since $z_i=(\log i)/\kappa$.
\item \underline{Control on $\Omega_2^{(i)}$.}
Let $i>1$.
Note that
\begin{equation}\label{Omega2c}
    \p\left[(\Omega_2^{(i)})^c\right]
\leq
    \p\left[b_i> T^{\uparrow}(f_i)-1\right]
    +\p\left[T^{\uparrow}(f_i)-1>i\, e^{\kappa f_i}\right]\, .
\end{equation}
Also, by definition of $e_j$, $H_j$, $\sigma(\cdot)$ and $b_i$,
\begin{eqnarray*}
    V^{\uparrow}(b_i)
& = &
    \max(H_j,\  0\leq j\leq \sigma(i)-1)
\\
& = &
    \max[\max(H_{\sigma(j)}, \ 1\leq j\leq i-1) , \max(H_j, \  \sigma(i-1)+1\leq j\leq \sigma(i)-1)],
\end{eqnarray*}
and recall that $\max(H_j, \  \sigma(i-1)+1\leq j\leq \sigma(i)-1)< f_i$ by definition of $\sigma(i)$.
Thus, by definition of $\sigma(\cdot)$, $f_i$ and $N_i$, and by \eqref{Igelhart0}, for some constant $C_5>0$,
\begin{eqnarray}
    \p\left[b_i\ge T^{\uparrow}(f_i)\right]
& = &
    \p\left[V^{\uparrow}(b_i)\geq f_i\right]
\le
    \sum_{j=1}^{i-1}
    \p\left[H_{\sigma(j)}\geq f_{i}\right]
=
    \sum_{j=1}^{i-1}
    \p\left[H_1\geq f_{i}|H_1\geq f_{j} \right]
\nonumber\\
&\le&
    C_5 \sum_{j=1}^{i-1}
    \exp\left(-\kappa[f_{i}-f_{j}]\right)
=
    C_5 \sum_{j=1}^{i-1} \left(\frac{N_j}{N_i}\right)^{\kappa}\left(\frac ij\right)^{\kappa(1+\varepsilon)}
\nonumber\\
&\le&
    C_5\left(\frac{C_0}{C_0-1}\right)^{\kappa} \sum_{j=1}^{i-1}\left(\frac {j!}{i!}\right)^{1+\varepsilon}
\nonumber\\
&\le&
    C_5\left(\frac{C_0}{C_0-1}\right)^{\kappa}
    \left(
    \frac{1}{i^{1+\e}}+
    \sum_{j=1}^{i-2}\frac {1}{[(i-1)i]^{1+\varepsilon}}
    \right)
\nonumber\\
&\le&
    O\left(i^{-1-\varepsilon}\right)\, .
\label{controleb1}
\end{eqnarray}
Moreover, setting $K_i:=i e^{\kappa f_i }+ 1$, we have for $i$ large enough,
\begin{eqnarray}
    \p\left[T^{\uparrow}(f_i)>K_i\right]
&\le&
    \p\left[e_{\lfloor \frac{K_i}{2\E[e_1]}\rfloor}>K_i\right]
    +
    \p\left[T^{\uparrow}(f_i)>e_{\lfloor \frac{K_i}{2\E[e_1]}\rfloor}\right]
\nonumber\\
&\le&
    \p\left[e_{\lfloor \frac{K_i}{2\E[e_1]}\rfloor}-\E[e_{\lfloor \frac{K_i}{2\E[e_1]}\rfloor}]>  \frac{K_i}2\right]
    +
    \big(1-\p\left[H_1\geq f_i\right]\big)^{\lfloor\frac {K_i}{2\E[e_1]} \rfloor}\nonumber\\
&\le&
    \frac{  2\mbox{var}(e_1)}{K_i\E[e_1]}
    +
    \Big(1-C^{(0)}_Ie^{-\kappa f_i}\Big)^{\frac {K_i}{2\E[e_1]}-1}
\nonumber\\
& \le &
    O(K_i ^{-1})+\exp\bigg(-\frac{C^{(0)}_Ie^{-\kappa f_i}(K_i-2\E[e_1])}{2\E[e_1]}\bigg)
\nonumber\\
&=&
    O\bigg(K_i^{-1}+\exp\bigg(-\frac{C^{(0)}_I i}{4\E[e_1]}\bigg)\bigg),
\label{controleb2}
\end{eqnarray}
where we used the fact that $e_{k+1}-e_k$, $k\geq 0$ are i.i.d.
so that
$\E(e_\ell)=\ell \E(e_1)$
and
$\mbox{var}(e_\ell)=\ell\, \mbox{var}(e_1)$ for $\ell\in\N$, and
$\{T^{\uparrow}(f_i)>e_\ell\}=\cap_{k=0}^{\ell-1}\{H_k<f_i\}$
where $H_k$, $k\in\N$ are also i.i.d.,
the Bienaym\'e-Chebychev inequality since $E(e_1^2)<\infty$ (as explained after \eqref{eq_Large_Dev_RW}),
$1-x\leq e^{-x}$ for $x\in\R$, and \eqref{Igelhart0}.
The quantities appearing in \eqref{controleb2} and
in \eqref{controleb1} are summable (since $\varepsilon>0$)
and so, due to \eqref{Omega2c}, $\sum_{i\ge 1}\p\left[(\Omega_2^{(i)})^c\right]<\infty$.

\item \underline{Control on $\Omega_3^{(i)}$.}
Let $i\geq 1$.
Observe that in every case, $a_i\geq \sup\{k< b_i\, :\, V(k)\geq V(b_i)+f_i+z_i\}$,
and so in particular, $V(a_i+1)<V(b_i)+f_i+z_i$. Since $V(b_i)=V(e_{\sigma(i)})=\min_{[0, e_{\sigma(i)+1}-1]}V$, this leads to
$$
    \forall 0\leq k\leq e_{\sigma(i)+1}-1,
\qquad
    V(k)
\geq
    V(b_i)
>
    V(a_i+1)-f_i-z_i.
$$
Notice that on $\big(\Omega_3^{(i)}\big)^c$, we have $c_i-a_i> C_4(f_i+z_i)$, and so
$$
    a_i+\lfloor C_4(f_i+z_i)\rfloor
\leq
    c_i
\leq
    e_{\sigma(i)+1}-1.
$$
Hence,
$
    \big(\Omega_3^{(i)}\big)^c
\subset
    \{V[a_i+\lfloor C_4(f_i+z_i)\rfloor]\geq V(a_i+1)-f_i-z_i\}
$.
Consequently, summing over all the possible values of $a_i+1$ when it is less than or equal to $K_i+1=i e^{\k f_i}+1$,
\begin{eqnarray}
    \p\Big[\big(\Omega_3^{(i)}\big)^c\cap \Omega_2^{(i)}\Big]
&\le&
    \p\big( V[a_i+\lfloor C_4(f_i+z_i)\rfloor ]\geq V(a_i+1)-f_i-z_i,\ a_{i}\leq b_i\leq K_i\big)
\nonumber\\
&\le&
    \sum_{x=1}^{K_i+1}\p\Big[V\Big(x+\lfloor C_4(f_i+z_i)\rfloor-1\Big)-V(x)\ge -f_i-z_i\Big]
\nonumber\\
&\le&
    (K_i+1)\p\big[V\big(\lfloor C_4(f_i+z_i)\rfloor-1\big)+f_i+z_i \geq 0\big]
\nonumber
\end{eqnarray}
by the Markov property. Hence, using $\p(X\geq 0)=\p[e^{\k_0 X}\geq 1]\le \E[e^{\k_0 X}]$
(due to the Markov inequality since $\k_0>0$), Remark \ref{expo<1}
and $\E\big(e^{\k_0V(n)}\big)=\big[\E\big(e^{\k_0V(1)}\big)\big]^n=v_0^n$ for $n\geq 1$, we get
\begin{eqnarray}
    \p\Big[\big(\Omega_3^{(i)}\big)^c\cap \Omega_2^{(i)}\Big]
&\le&
    2K_i\E\Big[\exp\Big(\kappa_0 \Big(V\big(\lfloor C_4(f_i+z_i)\rfloor-1\big)+f_i+z_i\Big)\Big)\Big]
\nonumber\\
&\le& \
    2 i e^{\k f_i}
    e^{\k_0(f_i+z_i)}
    v_0^{\lfloor C_4(f_i+z_i)\rfloor-1}
\leq
    2 i e^{(\k+\k_0) (f_i+z_i)}
    v_0^{C_4(f_i+z_i)-2}
\nonumber\\
&\le& \
    2v_0^{-2} i e^{(\k+\k_0+C_4(\log v_0)) (f_i+z_i)}\, .
\nonumber
\end{eqnarray}
Since $f_i+z_i\geq i$ for large $i$, $v_0<1$  and  $C_4>\frac{2(\kappa+\k_0)}{|\log v_0|}$, this gives
\begin{equation}
    \p\Big[\big(\Omega_3^{(i)}\big)^c\Big]
\le
    \p\Big[\big(\Omega_2^{(i)}\big)^c\Big]
    +
    \p\Big[\big(\Omega_3^{(i)}\big)^c\cap \Omega_2^{(i)}\Big]
\leq
    \p\Big[\big(\Omega_2^{(i)}\big)^c\Big]
    +
    O\big( i e^{-(\k+\k_0)i}\big)\, .
\label{c-b}
\end{equation}
Thanks to the control on $\Omega_2^{(i)}$ and since $\k+\k_0>0$,
this gives $\sum_{i\ge 1}\p\left[(\Omega_3^{(i)})^c\right]<\infty$.

\item
\underline{Control on $\Omega_4^{(i)}$.}
Now let us prove that $\sum_{i\ge 1}\p\left((\Omega_4^{(i)})^c\cap \Omega_6^{(i)}\right)<\infty$.
We set for $x\leq y$ and $h>0$,
\begin{eqnarray*}
&&
    E^{\uparrow}_{x,y}(h):=
\\
&  &
    \left\{\inf\{k>0 : V(x+k)\geq V(x)+h\}\leq \min(y-x,\inf\{k>0 : V(x+k)\le V(x)\})\right\} .
\end{eqnarray*}
This event corresponds to the case where, after location $x$, the potential increases at least $h$ before $y$ and before becoming again smaller than
or equal to its value at $x$.
Let $i\geq 1$.
Note that
$$
    \big(\Omega_4^{(i)}\big)^c\subset \mathcal A_i\cup\mathcal B_i \, ,
$$
with $\mathcal A_i:=\cup_{a_i<u<\alpha_i}\{V(u)\leq V(b_i)+f_i/4\}$
and $\mathcal B_i:=\cup_{\gamma_i<u<c_i}\{V(u)\leq V(b_i)+f_i/4\}$.
Notice that if $\mathcal A_i\neq\emptyset$, then $\alpha_i>a_i\geq  e_{\sigma(i-1)+1}$,
so $V(\alpha_i)\geq V(b_i)+f_i/2$.
On $\mathcal A_i$, denote by $x$ the largest $u$ satisfying the conditions defining $\mathcal A_i$.
Starting from $x$, the potential hits $V(\alpha_i)\in[V(b_i)+f_i/2,+\infty[\subset[V(x)+f_i/4, +\infty[$
before going back to $]-\infty, V(x)]$
and before $b_i$ since $\alpha_i<b_i$.
Then from $b_i=e_{\sigma(i)}$, the potential hits $[V(b_i)+f_i,+\infty[$ before going back to $]-\infty, V(b_i)]$,
because $H_{\sigma(i)}\geq f_i$, and before $c_i$.

Similarly on $\mathcal B_i\cap \Omega_6^{(i)}$, denote by $y$ the largest $u$ satisfying the conditions defining $\mathcal B_i$.
Notice that $V(c_i)\geq V(b_i)+f_i +z_i\geq V(y)+3f_i/4$ since $\omega\in \Omega_6^{(i)}$.
Starting from $b_i$, the potential hits $V(\gamma_i)\in[V(b_i)+f_i/2,+\infty[$ at location $\gamma_i$ (and so before $y$),
and before going back to $]-\infty, V(b_i)]$. Also, starting from $y>\gamma_i$, it hits
$[V(b_i)+f_i+z_i,+\infty[\subset [V(y)+3f_i/4,+\infty[$
before going back to $]-\infty, V(y)]$ and before $c_i$.
Hence,
$$
    \mathcal A_i
\subset
    \bigcup_{x\in]a_i,\alpha_i[} E^{\uparrow}_{x,b_i}\left(\frac {f_i}4\right)\cap
   E^{\uparrow}_{b_i,c_i}(f_i)
\ \mbox{ and }\
    \mathcal B_i\cap \Omega_6^{(i)}
\subset
    \bigcup_{y\in]\gamma_i,c_i[} E^{\uparrow}_{b_i,y} \left(\frac {f_i}2\right)\cap E^{\uparrow}_{y,c_i}\left(\frac {3f_i}4\right).
$$
Since $a_i$, $\alpha_i$, $b_i$ and $\gamma_i$ are not stopping times, we sum over their possible values,
using $b_i\leq i e^{\k f_i}=K_i-1$ and $c_i-a_i\leq C_4(f_i+z_i)$ on $\Omega_2^{(i)}\cap  \Omega_3^{(i)}$
and $0\leq a_i\leq \alpha_i\leq b_i\leq \gamma_i\leq c_i$ as follows.
For large $i$, by the Markov property at times $b$ and $y$, followed by \eqref{Igelhart0},
\begin{eqnarray*}
&\ &
    \p\left[\big(\Omega_4^{(i)}\big)^c\cap \Omega_6^{(i)}\cap \Omega_2^{(i)}\cap  \Omega_3^{(i)}\right]
\\
&\le&
    \sum_{0\leq x\le K_i}\, \sum_{b\in]x,x+C_4(f_i+z_i)]}
        \p\left[  E^{\uparrow}_{x,b}\left(\frac {f_i}4\right)\cap E^{\uparrow}_{b,K_i+C_4(f_i+z_i)}(f_i)\right]
\\
&&
\quad
    +
    \sum_{0\leq b\le K_i}\, \sum_{y\in]b,b+C_4(f_i+z_i)]}
        \p\left[ E^{\uparrow}_{b,y} \left(\frac {f_i}2\right)\cap E^{\uparrow}_{y,K_i+C_4(f_i+z_i)}\left(\frac {3f_i}4\right)\right]
\\
& \le &
    2K_i\, C_4(f_i+z_i)
    \left(\p\left[H_1\geq \frac {f_i}4\right]\p\left[H_1\geq {f_i}\right]
    +\p\left[H_1\geq \frac {f_i}2\right] \p\left[H_1\geq \frac {3f_i}4\right]\right)
\\
& \le &
    O\left(K_i f_i e^{-\frac{5\kappa f_i}4}\right)
    =O\left(i f_i e^{-\frac{\kappa f_i}4}\right)\, .
\end{eqnarray*}
From this, we conclude that $\sum_{i\ge 1}\p\left((\Omega_4^{(i)})^c\cap \Omega_6^{(i)}\right)<\infty$,
thanks to \eqref{eq_def_f_i} and to the controls on
$\Omega_2^{(i)}$ and $\Omega_3^{(i)}$.
\end{itemize}
\end{proof}
We now turn to $\Omega_5^{(i)}$ and $\Omega_6^{(i)}$.
\begin{prop}\label{BorelCantelli2}
We have
$$\sum_{i\ge 1}\p\left[\Omega_5^{(i)}\cap\Omega_6^{(i)}\right]=\infty\, . $$
\end{prop}
\begin{proof}
We have for $i\geq 1$, due to the strong Markov property at stopping time $\beta_i^+$,
\begin{eqnarray*}
    \p\left[\Omega_5^{(i)}\cap\Omega_6^{(i)}\right]
&\ge &
    \p\left[\Omega_5^{(i)}\cap\Omega_6^{(i)}\cap \{\beta_i^-\ge e_{\sigma(i-1)+1}\}\right]
\\
&\ge&
    \p\left[\sum_{k=\beta_i^-}^{\beta_i^+}e^{-(V(k)-V(b_i))}<7C_{2},\ \beta_i^- \ge e_{\sigma(i-1)+1}\right]
    \p\left[H_1>z_i\right]
\\
&\ge&\
    \p\left[\sum_{k={}^{\uparrow}T_{f_i}(f_i)}^{T^{\uparrow}(f_i)}e^{-[V(k)-V(m_1(f_i))]}<7C_{2}
            ,\  {}^{\uparrow}T_{f_i}(f_i)\ge 0\right]
    C_I^{(0)}e^{-\kappa z_i}\, ,
\end{eqnarray*}
where we used the Markov property at stopping time $e_{\sigma(i-1)+1}$ and \eqref{Igelhart0} in the last line.  Finally, due to the Markov inequality and to Lemma \ref{LemmaConstantInvariantMeasure}, we obtain
\begin{eqnarray*}
&\ &
    \p\left[\Omega_5^{(i)}\cap\Omega_6^{(i)}\right]
\\
&\ge&
    \p\left[\sum_{k={}^{\uparrow}T_{f_i}(f_i)}^{T^{\uparrow}(f_i)}e^{-[V(k)-V(m_1(f_i))]}<7C_{2}
        \left|  {}^{\uparrow}T_{f_i}(f_i)\ge 0\right.\right]
        \p\left[ {}^{\uparrow}T_{f_i}(f_i)\ge 0\right]C_I^{(0)}e^{-\kappa z_i}
\\
& = &
    \left(1-\p\left[\sum_{k={}^{\uparrow}T_{f_i}(f_i)}^{T^{\uparrow}(f_i)}e^{-[V(k)-V(m_1(f_i))]}\geq 7C_{2}\left|  {}^{\uparrow}T_{f_i}(f_i)\ge 0\right.\right]\right)
    \p\left[ {}^{\uparrow}T_{f_i}(f_i)\ge 0\right]C_I^{(0)}e^{-\kappa z_i}
\\
&\ge&
    \left(1- \frac{\E\left[\sum_{k={}^{\uparrow}T_{f_i}(f_i)}^{T^{\uparrow}(f_i)}e^{-[V(k)-V(m_1(f_i))]}
    \left|  {}^{\uparrow}T_{f_i}(f_i)\ge 0\right.\right]}{7C_{2}}\right)
    \p\left[ {}^{\uparrow}T_{f_i}(f_i)\ge 0\right]
    C_I^{(0)}e^{-\kappa z_i}
\\
&\ge&
    \left(1- 2/7\right)\p\left[
    {}^{\uparrow}T_{f_i}(f_i)\ge 0\right]
    C_I^{(0)} e^{-\kappa z_i} \, ,
\end{eqnarray*}
for all $i$ large enough,
where we used the fact that $(V(x+m_1(f_i))-V(m_1(f_i)),\ x\geq 0)$ is,
by Proposition \ref{LemmaTanakaStyleForRW} and the strong Markov property at stopping time $T^{\uparrow}(f_i)$,
independent of
$(V(x+m_1(f_i))-V(m_1(f_i)),\ x\leq 0)$ in the last inequality while applying Lemma \ref{LemmaConstantInvariantMeasure}.
Moreover, due to Proposition \ref{BorelCantelli1} and once more the strong Markov property,
$$
    \p\left[ {}^{\uparrow}T_{f_i}(f_i)\ge 0\right]
\geq
    \P\big(\Omega_1^{(i)}\big)
\to_{i\rightarrow +\infty}
    1.
$$
The result follows then from the fact that
$e^{-\kappa z_i}=1/i$.
\end{proof}
\begin{prop}\label{Prop_i_n}
The set $\tilde\Omega$ of environments such that
there exists a strictly increasing sequence $(i(n))_{n\in\N}$ of integers satisfying the following properties for every $n\in\N$ has
probability
$\p\big(\tilde\Omega\big)=1$:
\begin{equation}\label{lowerboundi(n)}
i(n)\ge\max\left(n,n^{\frac{1+\varepsilon}{1/\kappa-1-3\varepsilon/2}}\right)\, ,
\end{equation}
\begin{equation}\label{monteeGauche}
a_{i(n)}=\sup\{k< b_{i(n)}\, :\, V(k) \geq V(b_{i(n)})+f_{i(n)} +z_{i(n)}\}\, ,
\end{equation}
\begin{equation}\label{inegalphabeta}
    c_{i(n)}-a_{i(n)}
\le
    C_4(f_{i(n)}+z_{i(n)})\, ,
\end{equation}
\begin{equation}\label{Premieremontee}
b_{i(n)}\le T^{\uparrow}(f_{i(n)})-1\, ,
\end{equation}
\begin{equation}\label{Vminore}
\inf_{]a_{i(n)},c_{i(n)}[\setminus]\alpha_{i(n)},\gamma_{i(n)}[}V> V(b_{i(n)})+\frac{f_{i(n)}}4\, ,
\end{equation}
\begin{equation}\label{sumvalleybound}
\sum_{k=a_{i(n)}}^{c_{i(n)}}e^{-(V(k)-V(b_{i(n)}))}<8C_{2}\, ,
\end{equation}
\begin{equation}\label{eqDefHeightin}
H_{\sigma(i(n))}\ge f_{i(n)}+z_{i(n)}.
\end{equation}
\end{prop}
Note that $\frac{1}{\kappa} - 1 -\frac{3\varepsilon}{2} > 0$ since $\varepsilon\in]0,\frac{1-\kappa}{2\kappa}[$
and $0<\k<1$, see the beginning of Section \ref{Sect_Valleys}.
The valley number $i(n)$, that is, $[a_{i(n)}, c_{i(n)}]$, $n\in\N$, is called the $n$-th{ \it very deep valley}.

\begin{proof}
Due to Proposition \ref{BorelCantelli1} and to the Borel-Cantelli lemma, $\p$-a.s., there exists
$i_0\ge 1$ (we denote by $i_0$ the smallest one)
such that for every $i\ge i_0$, 
$\omega\in \cap_{j=1}^4 \Omega_j^{(i)}$ or $\omega\in \big(\Omega_6^{(i)}\big)^c$.
That is, for every $i\geq i_0$, $\omega\in \big(\Omega_6^{(i)}\big)^c$ or the following holds true
$$
    a_i=\sup\{k< b_i\, :\, V(k) \geq V(b_i)+f_i +z_i\},\qquad
    c_i-a_i\le C_4(f_i+z_i)\, ,$$
$$
    b_{i}\le T^{\uparrow}(f_{i})-1\le ie^{\kappa f_i}\, ,
\qquad
    \inf_{]a_i,c_i[\setminus]\alpha_i,\gamma_i[}V > V(b_i)+\frac {f_i}4\, ,
$$
so that $\omega\in \big(\Omega_6^{(i)}\big)^c$ or
$$
    \sum_{k\in]a_i,c_i[\setminus]\alpha_i,\gamma_i[} e^{-(V(k)-V(b_i))}
\le
    C_4(f_i+z_i) e^{-f_i/4}
\leq
    C_{2} \, ,
$$
for every $i$ large enough, and so for all $i\geq i_0$ (up to a change of the value of $i_0$).
Moreover, due to Remark \ref{indep}, the events $\Omega_5^{(i)}\cap\Omega_6^{(i)}$, $i\geq 1$ are independent.
So, Proposition \ref{BorelCantelli2} and the Borel-Cantelli lemma
ensure that, $\p$-almost surely, the set $\mathcal I $ (depending on the environment) of positive integers $i\ge i_0$ such that:
$$
    \sum_{k= \alpha_i}^{\gamma_i}e^{-(V(k)-V(b_i))}<7C_{2}\quad \mbox{and}\quad H_{\sigma(i)}\ge f_i+z_i
$$
is infinite.
We construct now $(i(n))_{n\in\N}$ by induction as follows:
$i(0):=\inf \mathcal I$ and for $n\geq 1$,
\begin{equation}\label{eq_def_i_n}
    i(n)
:=
    \inf\{j\ge \max(i(n-1)+1,n,n^{\frac{1+\varepsilon}{1/\kappa-1-3\varepsilon/2}})\ :\ j\in\mathcal I\}\, .
\end{equation}
By construction $(i(n))_{n\in\N}$ is strictly increasing, satisfies \eqref{lowerboundi(n)}-\eqref{eqDefHeightin} for every $n\in\N$,
and $i(n)\in\N$ for every $n\in\N$ on a event $\tilde \Omega$ having probability $1$.
\end{proof}

In the rest of the paper, for every $n\in\N$, $i(n)$ denotes the random variable (uniquely) defined by \eqref{eq_def_i_n},
and depending only on the environment.

\section{Quenched estimates for the RWRE}\label{Sect_Quenched}
This section is devoted to the proof of the main result, Theorem \ref{Theorem1Meeting}. We first need to prove some
preparatory lemmas, for which we use quenched techniques.
See Figure \ref{figure2_Potential_Valley_Transient} for a schema of the potential for a very deep valley.

\subsection{Hitting times and exit times of very deep valleys}
For any $y\in\mathbb Z$, recall that we write $\tau(y)$ for the time of the first visit
of $S$ to $y$:
$$
    \tau(y):=\inf\{k\ge 0\ :\ S_k=y\}.
$$
Also for $x\in\Z$, let $\P^x(\cdot):=\int P_\omega^x(\cdot)\p(\textnormal{d}\omega)$
be the annealed law of a RWRE in the environment $\omega$, starting from $x$.
The first lemma of this subsection says that for large $n\in\N$, a RWRE in the environment $\omega$
hits relatively quickly the bottom $b_{i(n)}$ of the $n$-th very deep valley.
\begin{lem} \label{LemmaHittingTimeVeryDeepValleys}
Let $x\in\Z$ and recall \eqref{eq_def_N_i}. Then,
$\PP^x$-almost surely,
\begin{equation}\label{eq:HittingTimeVeryDeepValleys}
\exists n_0\geq 1,\quad \forall n\ge n_0,\quad
    \tau(b_{i(n)})\leq N_{i(n)}/10\, .
\end{equation}
\end{lem}
We will only use the weaker statement.
\begin{cor} Let $x\in\Z$.
For $\p$-almost every environment $\omega$,
\begin{equation}\label{eq:HittingTimeVeryDeepValleys2}
  P_\omega^x[\tau(b_{i(n)})\le N_{i(n)}/10]\to_{n\to+\infty}1\, .
  \end{equation}
\end{cor}
\begin{proof}
 Clearly, \eqref{eq:HittingTimeVeryDeepValleys2} follows from \eqref{eq:HittingTimeVeryDeepValleys} since
 \eqref{eq:HittingTimeVeryDeepValleys} says that for almost every environment $\omega$,
$\mathbf 1_{\{\tau(b_{i(n)})\le N_{i(n)}/10\}}\to_{n\to+\infty}1$, $P_\omega^x$-almost surely, and taking expectations gives \eqref{eq:HittingTimeVeryDeepValleys2}.
\end{proof}
\begin{proof}[Proof of Lemma \ref{LemmaHittingTimeVeryDeepValleys}]
We start with the case $x=0$.
We consider the RWRE $S'$ reflected at 0, obtained by deleting the excursions of $S$ in $\mathbb Z\setminus \mathbb N$.
More precisely, let
$r_0:=0$ and
$$
    r_{k+1}
:=
    \min\{\ell> r_k\ :\ S_\ell\ge 0,\ S_{\ell-1}\ge 0\},\quad k\in\N.
$$
We define formally $S':=(S'_k)_{k\in\N}$ by $S'_k:=S_{r_k}$ for every $k\in \mathbb N$. Its hitting times are denoted by
$$
    \tau'(y)
:=
    \inf\{k\ge 0\ :\ S'_k=y\},\quad y\in\mathbb N.
$$
Notice that, under $P_\omega^0$, $S'$ is a RWRE in the environment $\omega$, reflected at 0 on the right side
(that is, a RWRE with transition probabilities
given by \eqref{probatransition} in which we replace $\omega_0$ by 1 and $S^{(1)}$ by $S'$).
So $S'$ has, under $\P$, the same law as $S$ under $\PP_{|0}$ (defined in Fact \ref{FactTimeSpnetBeforeFirstValley}).

Consequently, for every positive integer $i$, due to the Markov inequality and to Fact \ref{FactTimeSpnetBeforeFirstValley},
we have ($N_i$ and $f_i$ being defined respectively in \eqref{eq_def_N_i} and \eqref{eq_def_f_i}),
\begin{eqnarray*}
    \PP\big[\tau'\big(T^{\uparrow}(f_{i})-1\big)>N_{i}/40\big]
&\leq&
    40 N_i^{-1}\, \EE\big[ \tau'\big(T^{\uparrow}(f_{i})-1\big)\big]
\\
&=&
    40N_i^{-1}\, \EE_{|0}\big[\tau(T^{\uparrow}(f_{i})-1\big)\big]
\\
&\leq&
    40N_i^{-1}\, C_0 e^{f_{i}}
\\
& = &
    40i^{-1-\varepsilon}\, .
\label{eqAvantdefinitiondeg}
\end{eqnarray*}
This gives
$ \sum_{i\geq 1} \PP[\tau'(T^{\uparrow}(f_{i})-1)>N_{i}/40]
< \infty$.
Hence,
due to the Borel-Cantelli lemma, $\PP$-almost surely, for all $i$ large enough,
$\tau'(T^{\uparrow}(f_{i})-1)\leq N_{i}/40$.
Moreover,
$$
    \tau\big(T^{\uparrow}(f_{i})-1\big)-\tau'\big(T^{\uparrow}(f_{i})-1\big)
\le
    \#\{k\ge 0\ :\ S_k\le 0,\ S_{k+1}\le 0\}<\infty\,
\quad
    \PP-\mbox{a.s.}\, ,
$$
since, due to \cite{S75}, $S_k\rightarrow+\infty$ $\PP$-almost surely as $k$ goes to infinity
because $\E(\log \rho_0)<0$ by \eqref{rho<0} in our setting.
Therefore  $\PP$-almost surely, for all $i$ large enough,
$\tau(T^{\uparrow}(f_{i})-1)\leq N_{i}/20$ and, in particular,
for $n$ large enough,
$\tau(b_{i(n)})\le \tau(T^{\uparrow}(f_{i(n)})-1)\leq N_{i(n)}/20$
due to \eqref{Premieremontee}.
This proves the lemma in the case $x=0$.

We now assume that $x<0$.
We use the simple decomposition $\tau(b_{i(n)})=\tau(0)+[\tau(b_{i(n)})-\tau(0)]$.
Since the RWRE $S$ is transient to $+\infty$,
$\tau(0)<\infty$ $\P^x$-almost surely, and so $\tau(0)\leq N_{i(n)}/20$ for large $n$, $\P^x$-almost surely.
Also, for $P$-almost every $\omega$, $(S_{k+\tau(0)}-S_{\tau(0)})_{k\in\N}$ has under $P_\omega^x$
the same law as $S$ under $P_\omega^0$ by the strong Markov property, so we can apply to it the results of the case $x=0$.
Hence, the hitting time of $b_{i(n)}$ by $(S_{k+\tau(0)}-S_{\tau(0)})_{k\in\N}$,
that is, $\tau(b_{i(n)})-\tau(0)$, is less than $N_{i(n)}/20$ for large $n$, $P_\omega^x$-almost surely, for $\p$-almost every $\omega$.
Summing these two inequalities proves the lemma in the case $x<0$.

Finally, assume that $x>0$. We use a simple coupling argument.
Possibly in an enlarged probability space, we can consider our RWRE $S$ starting from $x$ in the environment $\omega$,
and a RWRE $Z$ in the same environment $\omega$ but starting from $0$,
independent of $S$ conditionally on $\omega$. That is, for every $\omega$, $(S,Z)$
has the law of $\big(S^{(1)},S^{(2)}\big)$ under $P_{\omega}^{(x,0)}$ as defined in \eqref{probatransition}, so with a slight abuse of notation, we denote by $P_{\omega}^{(x,0)}$ the law of $(S,Z)$ conditionally on $\omega$.
Denote by $\tau_Z(x)$ the hitting time of $x$ by $Z$, which is finite almost surely since
$Z_0=0<x$ and $Z$ is transient to $+\infty$.
We can now define $Z'_k:=Z_k$ if $k\leq \tau_Z(x)$ and $Z'_k=S_{k-\tau_Z(x)}$ otherwise.
Conditionally on $\omega$, $Z'$ is, under $P_{\omega}^{(x,0)}$, a RWRE in the environment $\omega$, starting from $0$.
Hence we can apply to it and to its hitting times $\tau_{Z'}$ the previous result for $x=0$,
and so $\tau(b_{i(n)})=\tau_S(b_{i(n)})\leq \tau_{Z'}(b_{i(n)})\leq N_{i(n)}/20$ for large $n$, $P_{\omega}^{(x,0)}$-almost surely for $\p$-almost every $\omega$,
which proves the lemma when $x>0$.
\end{proof}

The next lemma says that with large probability, after first hitting $b_{i(n)}$, each RWRE
$S^{(j)}$ stays a long time between $a_{i(n)}$
and $ c_{i(n)}$.
\begin{lem} \label{LemmaStayingInside}
For every $\omega\in\tilde\Omega$,
$$
    \lim_{n\to+\infty}
    P_\omega^{b_{i(n)}}\big[\tau\big( a_{i(n)}-1\big)\wedge \tau\big( c_{i(n)}+1\big)<2N_{i(n)}\big]
=
    0.
$$
\end{lem}

\begin{proof}
Let $\omega\in\tilde\Omega$ and $n\geq 1$.
Observe that $0\leq e_{\sigma(i(n)-1)+1}\leq a_{i(n)} < b_{i(n)} <  c_{i(n)} \leq   e_{\sigma(i(n))+1}$ by \eqref{monteeGauche} and \eqref{eqDefHeightin}.
Hence, $\min_{[ a_{i(n)},  c_{i(n)}]}V=V(b_{i(n)})$.
Also, recall that $|V(y)-V(y-1)|\leq \log(\e_0^{-1})$ for any $y\in\Z$ by \eqref{UE}.
So by Fact \ref{FactInegProbaAtteinte},
\begin{eqnarray}
&&
    \po^{b_{i(n)}}\big[\tau\big( a_{i(n)}-1\big)\wedge \tau\big( c_{i(n)}+1\big)< 2 N_{i(n)}\big]
\nonumber\\
& \leq &
    2 (2N_{i(n)}) \varepsilon_0^{-1}\exp[-(f_{i(n)}+z_{i(n)})]
\nonumber\\
& = &
    4 \varepsilon_0^{-1} N_{i(n)}
    \exp[-(\log N_{i(n)} -\log C_0+[1/\k-1-\e]\log i(n))]
\nonumber\\
& = &
    4 C_0 \varepsilon_0^{-1} (i(n))^{-(1/\k-1-\e)}.
\nonumber
\end{eqnarray}
Since $1/\k-1-\e>0$ (recall that $0<\k<1$ and $\varepsilon\in]0,\frac{1-\kappa}{2\kappa}[$, see the beginning of Section \ref{Sect_Valleys}), the lemma is proved.
\end{proof}



\subsection{Coupling}\label{SubSecEndofSketch}
Similarly as in \cite[Subsection 5.3]{Devulder_Gantert_Pene} (see also \cite{Brox} and \cite{AndreolettiDevulder} for diffusions in random potentials), we use a coupling between $S$ starting from $b_{i(n)}$, and a
reflected RWRE $\widehat S$, defined as follows.
For fixed $n$ and $\omega\in\tilde\Omega$, we define $\widehat\omega_{ a_{i(n)}}=1$, $\widehat\omega_{ c_{i(n)}}=0$
and $\widehat\omega_y=\omega_y$ every for $y\in\mathbb Z\setminus \{ a_{i(n)}, c_{i(n)}\}$.
For $x\in\{ a_{i(n)},\dots, c_{i(n)}\}$, we consider
a RWRE $\widehat S$ in the environment $\widehat \omega:=(\widehat \omega_y)_{y\in\Z}$, starting from $x$,
and denote its quenched law by $P_{\widehat \omega}^x$. So, $\widehat S$ satisfies \eqref{probatransition}
with $\widehat \omega$ instead of $\omega$ and $\widehat S$ instead of $S^{(1)}$.
That is, $\widehat S$ is a RWRE in the environment $\omega$, reflected at $ a_{i(n)}$ and $ c_{i(n)}$.
Now, define the measure $\widehat \mu_n$ on $\Z$ by
$$
    \widehat\mu_n( a_{i(n)}):=e^{-V( a_{i(n)})},\
\qquad
    \widehat\mu_n( c_{i(n)}):=e^{-V( c_{i(n)}-1)},
$$
$$
    \forall x\in] a_{i(n)}, c_{i(n)}[,
\qquad
           \widehat \mu_n(x):=e^{-V(x)}+e^{-V(x-1)}\, ,
$$
and $\widehat \mu_n(x):=0$ for all $x\notin[ a_{i(n)}, c_{i(n)}]$.
Note that, for fixed environment $\omega$ and fixed $n$,  $\widehat\mu_n/\widehat\mu_n(\Z)$ is the invariant probability measure of $\widehat S$
(indeed, a classical and direct computation shows that $\widehat\mu_n(x)=\widehat\omega_{x-1}\widehat\mu_n(x-1)+(1-\widehat\omega_{x+1})\widehat\mu_n(x+1)$ for all $x\in\Z$).
Thus, an invariant probability measure for $\big(\widehat S_{2k})_{k\in\N}$ is $\widehat \nu_n$, defined by
\begin{equation}\label{eq_def_nu}
    \widehat \nu_n(x)
:=
    \widehat\mu_n(x)\un_{2\mathbb Z+b_{i(n)}}(x)/\widehat\mu_n(2\mathbb Z+b_{i(n)}),
\qquad
    x\in\Z\, .
\end{equation}
That is, $P_{\widehat \omega}^{\widehat \nu_n}\big(\widehat S_{2k}=x\big)=\widehat\nu_n(x)$ for all $x\in\Z$ and $k\in\N$,
with $P_{\widehat \omega}^{\widehat \nu_n}(\cdot):=\sum_{x\in\Z}\widehat\nu_n(x)P_{\widehat \omega}^x(\cdot)$.

For fixed $n$ and $\omega\in\tilde\Omega$, we consider a coupling $Q_\omega^{(n)}$ of $S$ and $\widehat S$ such that:
\begin{equation}\label{eq_def_Q}
    Q_\omega^{(n)}\big(\widehat S\in\cdot\big)
=
    P_{\widehat\omega}^{\widehat\nu_n}\big(\widehat S\in\cdot\big),
\qquad
    Q_\omega^{(n)}(S\in\cdot)
=
    P_\omega^{b_{i(n)}}(S\in\cdot),
\end{equation}
such that under $Q_\omega^{(n)}$, $\widehat S$ and $S$ move independently
until
$$
    \tau_{\widehat S=S}
:=
    \inf\big\{k\ge 0\, :\,  \widehat S_k=S_k\big\} \, ,
$$
then $\widehat S_k=S_k$ for every $\tau_{\widehat S=S}\leq k< \tau_{exit}$, where
$$
    \tau_{exit}
:=
    \inf\big\{k> \tau_{\widehat S=S}\, :\,  S_k\not\in[ a_{i(n)}, c_{i(n)}]\big\}\, ,
$$
then $\widehat S$ and $S$ move independently again after $\tau_{exit}$.
We stress that $\widehat \omega$,
$\tau_{\widehat S=S}$ and $\tau_{exit}$ depend on $n$, but we do not write $n$ as a subscript.

In order to show that the coupling occurs quickly under $Q_\omega^{(n)}$, i.e. that $\tau_{\widehat S=S}$ is small,
we prove the two following lemmas.
As before, for $x\in\Z$, $\tau(x)$ denotes the hitting time of $x$ by $S$ (not by $\widehat S$).
\begin{lem}\label{atteintetilde}
For every $\omega\in\tilde\Omega$, we have
$$
    \lim_{n\rightarrow +\infty}Q_\omega^{(n)} \left(\max\left(\tau(\alpha_{i(n)}),\tau(\gamma_{i(n)})\right)>\frac {N_{i(n)}}{2}\right)
=
    0.
$$
\end{lem}
\begin{proof}
Let $\omega\in\tilde\Omega$.
Due to the Markov inequality and to \eqref{InegProba3} and \eqref{inegalphabeta}, we first observe that
\begin{eqnarray}
&\ &
    P_\omega^{b_{i(n)}}\left(\tau(\alpha_{i(n)})\wedge \tau(c_{i(n)})>\frac{N_{i(n)}}2\right)
\nonumber\\
&\le&
    (N_{i(n)}/2)^{-1}
    \varepsilon_0^{-1}(c_{i(n)}-\alpha_{i(n)})^2\exp\left(\max_{\alpha_{i(n)}\le \ell\le k\le c_{i(n)}-1,\ \ell\le b_{i(n)}-1}(V(\ell)-V(k))\right)
\nonumber\\
&\le&
    (N_{i(n)}/2)^{-1}\varepsilon_0^{-1}(2C_4\log N_{i(n)})^2
    \exp\left(\max_{[\alpha_{i(n)}, b_{i(n)}-1]}V-V(b_{i(n)})\right)
\nonumber\\
&\le&
    C_6 (N_{i(n)})^{-1}(\log N_{i(n)})^2\exp\left(f_{i(n)}/2\right)
\nonumber\\
&\le&
    C_7\frac{(\log N_{i(n)})^2}{(i(n))^{\frac{1+\varepsilon}{2}}\sqrt{N_{i(n)}}}
\label{AAAA1}
\end{eqnarray}
for every $n$ large enough since $V(\alpha_{i(n)})\leq V(b_{i(n)})+f_{i(n)}/2+\log(\e_0^{-1})$.
Moreover, due to \eqref{EgProba4} and \eqref{inegalphabeta},
for $n$ large enough
\begin{equation}\label{AAAA2}
    P_\omega^{b_{i(n)}}\left[\tau(\alpha_{i(n)})> \tau(c_{i(n)})\right]
\le
    \frac{\sum_{j=\alpha_{i(n)}}^{b_{i(n)}-1}e^{V(j)}}{\sum_{j=\alpha_{i(n)}}^{ c_{i(n)}-1}e^{V(j)}}
\le
    \frac{(2C_4\log N_{i(n)})\e_0^{-1}e^{\frac{f_{i(n)}}2}}{\e_0e^{f_{i(n)}+z_{i(n)}}}\, ,
\end{equation}
where we used the fact $\sum_{j=\alpha_{i(n)}}^{ c_{i(n)}-1}e^{V(j)}\ge e^{V(c_{i(n)}-1)}\ge \varepsilon_0 e^{V(b_{i(n)})+f_{i(n)}+z_{i(n)}}$ due to \eqref{eqDefHeightin} and \eqref{UE}.
Using \eqref{eq_def_Q}, then combining \eqref{AAAA1} and \eqref{AAAA2}, we obtain
\begin{equation}\label{eq_limit_1}
    Q_\omega^{(n)} \left(\tau(\alpha_{i(n)})>\frac {N_{i(n)}}{2}\right)
=
    P_\omega^{b_{i(n)}} \left(\tau(\alpha_{i(n)})>\frac {N_{i(n)}}{2}\right)
\to_{n\rightarrow +\infty}
    0.
\end{equation}
We prove analogously that
\begin{equation}\label{eq_limit_2}
    Q_\omega^{(n)} \left(\tau(\gamma_{i(n)})>\frac {N_{i(n)}}{2}\right)
=
    P_\omega^{b_{i(n)}} \left(\tau(\gamma_{i(n)})>\frac {N_{i(n)}}{2}\right)
\to_{n\rightarrow +\infty}
    0,
\end{equation}
using \eqref{InegProba4} instead of \eqref{InegProba3}.
Finally, \eqref{eq_limit_1} together with \eqref{eq_limit_2} prove the lemma.
\end{proof}

\begin{lem}\label{tempscouplage}
For every $\omega\in\tilde\Omega$, we have
$$
    \lim_{n\rightarrow +\infty}
    Q_\omega^{(n)}\left[\tau_{\widehat S=S}>\max\left(\tau(\alpha_{i(n)}),\tau(\gamma_{i(n)})\right)\right]=0\, .
$$
\end{lem}
\begin{proof}
Let $\omega\in\tilde\Omega$.
We fix $n\in\N$.
First, notice that due to \eqref{inegalphabeta} and \eqref{Vminore},
\begin{eqnarray}
    \widehat\nu_n\left([a_{i(n)},\alpha_{i(n)}]\right)
+
    \widehat\nu_n\left([\gamma_{i(n)},  c_{i(n)}]\right)
& \le &
    \sum_{k\in[ a_{i(n)}, c_{i(n)}]\setminus ]\alpha_{i(n)},\gamma_{i(n)}-1[}
    e^{-[V(k)-V(b_{i(n)})]}
\nonumber\\
& \leq &
    C_4 (f_{i(n)}+z_{i(n)})\, \varepsilon_0^{-1} e^{-\frac{f_{i(n)}}4}.
\label{Ineg_nu_loin}
\end{eqnarray}
Also, observe that
\begin{eqnarray}
&\ &
    Q_\omega^{(n)}\left[\tau_{\widehat S=S}>\max\left(\tau(\alpha_{i(n)}),\tau(\gamma_{i(n)})\right)\right]
\nonumber\\
&\le &
    Q_\omega^{(n)}\left(\tau_{\widehat S=S}>\tau(\alpha_{i(n)}),\ \widehat S_{0}<S_0\right)
    +
    Q_\omega^{(n)}\left(\tau_{\widehat S=S}>\tau(\gamma_{i(n)}),\ \widehat S_{0}>S_0\right).~~~~
\label{Ineg_Q}
\end{eqnarray}
Because of the definitions \eqref{eq_def_nu} and \eqref{eq_def_Q} of $\widehat\nu_n$ and $Q_\omega^{(n)}$,
$\widehat S_0-S_0=\widehat S_0-b_{i(n)}\in 2\mathbb Z$, $Q_\omega^{(n)}$-almost surely.
Moreover, the possible increments of $(\widehat S_k-S_k)_k$ are contained in $\{-2,0,2\}$. Therefore the walks $\widehat S$ and $S$ cannot cross  without meeting,
hence  $\widehat S_0<S_0,\ \tau_{\widehat S=S}> m$  implies for $m\ge 0$ that
$\widehat S_m<S_m$. So,
\begin{eqnarray}
    Q_\omega^{(n)}\left(\tau_{\widehat S=S}>\tau(\alpha_{i(n)}),\ \widehat S_{0}<S_0\right)
&\le&
    Q_\omega^{(n)}\left(\widehat S_{\tau(\alpha_{i(n)})}<S_{\tau(\alpha_{i(n)})},\ \tau_{\widehat S=S}>\tau(\alpha_{i(n)})\right)
\nonumber\\
&\le&
    Q_\omega^{(n)}\left(\widehat S_{2\lfloor \tau(\alpha_{i(n)})/2\rfloor}\leq\alpha_{i(n)},\ \tau_{\widehat S=S}>\tau(\alpha_{i(n)})\right)
\nonumber\\
&\le&
    \widehat\nu_n\left([ a_{i(n)},\alpha_{i(n)}]\right)
\nonumber\\
&\le&
    C_4\, \varepsilon_0^{-1}(f_{i(n)}+z_{i(n)})e^{-\frac{f_{i(n)}}4}\, ,
\label{Ineg_Q_nu}
\end{eqnarray}
where we used \eqref{Ineg_nu_loin} in the last inequality and
where the previous inequality comes from the independence of $\widehat S$ with $S$ and its hitting times $\tau(\cdot)$ up to
$\tau_{\widehat S=S}$ under $Q_\omega^{(n)}$, and from the fact that
$Q_\omega^{(n)}\big(\widehat S_{2k}=x\big)=P_{\widehat \omega}^{\widehat \nu_n}\big(\widehat S_{2k}=x\big)=\widehat\nu_n(x)$ for all $x\in\Z$ and $k\in\N$.
Analogously, we obtain that
\begin{equation}
    Q_\omega^{(n)}\left(\tau_{\widehat S=S}>\tau(\gamma_{i(n)}),\ \widehat S_{0}>S_0\right)
\le
    \widehat\nu_n\left([\gamma_{i(n)},  c_{i(n)}]\right)
\le
    C_4\, \varepsilon_0^{-1}(f_{i(n)}+z_{i(n)})e^{-\frac{f_{i(n)}}4}\, .
\label{Ineg_Q_nu2}
\end{equation}
Finally, combining \eqref{Ineg_Q}, \eqref{Ineg_Q_nu} and \eqref{Ineg_Q_nu2} proves the lemma.
\end{proof}

Finally, the following lemma says that for large $n$, with large $Q_\omega^{(n)}$ probability,
$\widehat S$ and $S$ are equal between times $N_{i(n)}/2$ and $2N_{i(n)}$:
\begin{prop}\label{couplage}
For every $\omega\in\tilde\Omega$, we have
$$
    \lim_{n\rightarrow +\infty}Q_\omega^{(n)}\left(\tau_{\widehat S=S}>\frac {N_{i(n)}}{2}\right)=0,
\qquad
    \lim_{n\rightarrow +\infty}Q_\omega^{(n)}\left(\tau_{exit}<2N_{i(n)}\right)=0\, .
$$
\end{prop}

\begin{proof}
The first limit is a direct consequence of Lemmas \ref{atteintetilde}  and \ref{tempscouplage}.
The second one follows from
\begin{eqnarray*}
    Q_\omega^{(n)}\left(\tau_{exit}<2N_{i(n)}\right)
& \leq &
    Q_\omega^{(n)}\big[\tau\big( a_{i(n)}-1\big)\wedge \tau\big(c_{i(n)}+1\big)<2N_{i(n)}\big]
\\
& = &
    P_\omega^{b_{i(n)}}\big[\tau\big(a_{i(n)}-1\big)\wedge \tau\big( c_{i(n)}+1\big)<2N_{i(n)}\big]
\to_{n\to+\infty}
    0
\end{eqnarray*}
by \eqref{eq_def_Q} and Lemma \ref{LemmaStayingInside}.
\end{proof}

\subsection{Conclusion}\label{Sect_conclusion}
The following lemma relies on the coupling of the previous subsection, and on the fact that
$\widehat\nu_n(b_{i(n)})$ is greater than some positive constant, uniformly in $n$ and $\omega$.

\begin{lem}\label{MINO1}
There exists $C_8>0$ such that, for every $\omega\in\widetilde\Omega$,
\begin{equation}
    \liminf_{n\rightarrow +\infty}\inf_{k\in [N_{i(n)}/2,2N_{i(n)}[\cap (2\mathbb Z)}
\ \
    P_\omega^{b_{i(n)}}\big[S_k=b_{i(n)}\big]
\ge
    C_8\, .
\end{equation}
\end{lem}
\begin{proof}
We fix $\omega\in\widetilde\Omega$.
First, notice that for every $n\in\N$, $\widehat\mu_n(2\mathbb Z+b_{i(n)})=\sum_{j= a_{i(n)}}^{ c_{i(n)}-1}e^{-V(j)}$, so
using \eqref{eq_def_nu} and \eqref{sumvalleybound},
\begin{equation}\label{eq_minoration_b}
    \forall n\in\N,
\qquad
    \widehat\nu_n(b_{i(n)})
\geq
    \frac{e^{-V(b_{i(n)})}}{\sum_{j= a_{i(n)}}^{ c_{i(n)}-1}e^{-V(j)}}
\geq
    \frac{1}{8C_2}
=:
    C_8>0.
\end{equation}
Using the definition of the coupling
(first \eqref{eq_def_Q}, and then $\widehat S_j=S_j$ for every $\tau_{\widehat S=S}\leq j< \tau_{exit}$),
we get for every $n\in\N$ and $k\in \big[\frac{N_{i(n)}}{2},2N_{i(n)}\big[$,
\begin{eqnarray*}
    P_\omega^{b_{i(n)}}\left[S_k=b_{i(n)}\right]
& = &
    Q_\omega^{(n)}\left[S_k=b_{i(n)}\right]
\\
&\ge&
    Q_\omega^{(n)}\left[S_k=b_{i(n)},\ \tau_{\widehat S=S}\leq \frac {N_{i(n)}}{2}\leq k< 2N_{i(n)}\leq \tau_{exit}\right]
\\
& = &
    Q_\omega^{(n)}\left[\widehat S_k=b_{i(n)},\
    \tau_{\widehat S=S}\leq \frac {N_{i(n)}}{2}\leq k< 2N_{i(n)}\leq \tau_{exit}
    \right]
\\
&\ge&
    Q_\omega^{(n)}\left[\widehat S_k=b_{i(n)}\right]
    -Q_\omega^{(n)}\left[\tau_{\widehat S=S}>\frac {N_{i(n)}}{2}\ \mbox{or}\ \tau_{exit}<2N_{i(n)}\right].
\end{eqnarray*}
Since
$
    Q_\omega^{(n)}\big[\widehat S_k=b_{i(n)}\big]
=
    P_{\widehat\omega}^{\widehat\nu_n}\big[\widehat S_k=b_{i(n)}\big]
=
    \widehat\nu_n\big(b_{i(n)}\big)
$
for every even $k$
(by \eqref{eq_def_Q} and due to the remark after \eqref{eq_def_nu}), and using Proposition \ref{couplage},
we get
$$
    \liminf_{n\rightarrow +\infty}\inf_{k\in [N_{i(n)}/2,2N_{i(n)}[\cap (2\mathbb Z)}
    P_\omega^{b_{i(n)}}\left[S_k=b_{i(n)}\right]
\geq
    \liminf_{n\rightarrow +\infty} \widehat\nu_n\big(b_{i(n)}\big)
\geq
    C_8
$$
due to \eqref{eq_minoration_b}.
Since this is true for every $\omega\in\widetilde\Omega$, this proves the lemma.
\end{proof}
We now set $N'_{i(n)}=N_{i(n)}-\mathbf 1_{\{(N_{i(n)}-b_{i(n)})\in (1+2\mathbb Z)\}}$, for all $n\in\N$ and $\omega\in\widetilde\Omega$,
so that $N'_{i(n)}$ and $b_{i(n)}$ have the same parity.
\begin{lem}\label{Lemma_Proba_etre_en_b}
Let $x\in(2\Z)$.
For $\p$-almost every environment $\omega$,
$$
    \liminf_{n\rightarrow +\infty} P_\omega^x\big[S_{N'_{i(n)}}=b_{i(n)}\big]
\ge
    \frac{C_8}2.
$$
\end{lem}
\begin{proof}
Define $\mathcal F_k:=\sigma(S_0, S_1, \dots, S_k, \omega)$ for $k\in\N$.
Let $x\in(2\Z)$ and $\omega\in\widetilde\Omega$.
Due to the strong Markov property,
\begin{eqnarray}
    P_\omega^x\big[S_{N'_{i(n)}}=b_{i(n)}\big]
&=&
    E_\omega^x\left[ P_\omega^x\left(S_{N'_{i(n)}}=b_{i(n)}\,\big|\,\mathcal F_{\tau(b_{i(n)})}\right)\right]
\nonumber\\
&\ge &
    E_\omega^x\left[ \mathbf 1_{\{\tau(b_{i(n)})\le N_{i(n)}/10\}}P_\omega^{b_{i(n)}}\left[S_k=b_{i(n)}\right]\bigg|_{k=N'_{i(n)}-\tau(b_{i(n)})}\right]
\nonumber\\
&\ge&
    \frac{C_8}{2}\, P_\omega^x\left[\tau(b_{i(n)})\le \frac {N_{i(n)}}{10}\right]\,
\label{Ineg_Proba_Egal_bin}
\end{eqnarray}
for large $n$,
since $\big(N'_{i(n)}-\tau(b_{i(n)})\big)\in[\frac{8 N_{i(n)}}{10},N_{i(n)}]\cap(2\Z)$ uniformly for large $n$ on $\{\tau(b_{i(n)})\le N_{i(n)}/10\}$,
and due to Lemma \ref{MINO1}.
Finally by \eqref{eq:HittingTimeVeryDeepValleys2}, $P_\omega^x[\tau(b_{i(n)})\le N_{i(n)}/10]\to_{n\to+\infty}1$ for $\p$-almost every environment $\omega$. This, combined with \eqref{Ineg_Proba_Egal_bin}, proves the lemma.
\end{proof}
\begin{proof}[Proof of Theorem \ref{Theorem1Meeting}]
Set $\mathcal D$ for the set of $(x_1,...,x_d)\in(2\mathbb Z)^d\cup(2\mathbb Z +1)^d$.
Let $(x_1,...,x_d)\in\mathcal D$.

First, notice that if $(x_1,...,x_d)\in(2\Z)^d$, for $\p$-almost every $\omega$, for $n$ large enough,
\begin{equation*}
    P_\omega^{(x_1,\dots,x_d)}\bigg[S_{N'_{i(n)}}^{(1)}=\ldots=S_{N'_{i(n)}}^{(d)}
    \bigg]
\geq
    \prod_{j=1}^d P_\omega^{x_j}\bigg[S_{N'_{i(n)}}=b_{i(n)}\bigg]
\geq
    \bigg(\frac{C_8}{4}\bigg)^d
=:
    C_9>0,
\end{equation*}
by independence of $S^{(1)},\dots,S^{(d)}$ under $P_\omega^{(x_1,\dots,x_d)}$ and thanks to Lemma \ref{Lemma_Proba_etre_en_b}.
As a consequence, if $(x_1,...,x_d)\in(2\Z+1)^d$, for $\p$-almost every $\omega$, for $n$ large enough, by the Markov property,
$$
    P_\omega^{(x_1,\dots,x_d)}\bigg[S_{N'_{i(n)}+1}^{(1)}=...=S_{N'_{i(n)}+1}^{(d)}\bigg]
=
    E_\omega^{(x_1,\dots,x_d)}
    \bigg[
        P_\omega^{(S_1^{(1)},\dots,S_1^{(d)})}\big(S_{N'_{i(n)}}^{(1)}=...=S_{N'_{i(n)}}^{(d)}\big)
    \bigg]
\geq
    C_9.
$$

For every nonnegative integer $n$, let us write $A_n$ and $E$ for the events:
$$
    A_n
:=
    \big\{S_{n}^{(1)}=\ldots = S_{n}^{(d)}\big\}
\quad \mbox{and}\quad
    E
:=
    \big\{A_n\ i.o.\big\}
=
    \cap_{N\ge 0}\cup_{n\ge N}A_n,
$$
and define $\theta_n:=N'_{i(n)}$ if $(x_1,...,x_d)\in(2\Z)^d$, and
$\theta_n:=N'_{i(n)}+1$ if $(x_1,...,x_d)\in(2\Z+1)^d$.
It follows from the two previous inequalities that, for every $(x_1,\dots,  x_d)\in\mathcal D$,
for $\p$-almost every $\omega$, for $n$ large enough,
$
    P_\omega^{(x_1,\dots,x_d)}\big[A_{\theta_n}\big]
\geq
    C_9
$.

As a consequence,
for every $(x_1,...,x_d)\in\mathcal D$,
for $\p$-almost every $\omega$,
\begin{eqnarray}
    \po^{(x_1,x_2,\dots,x_d)}\big[E\big]
&\ge&
    \po^{(x_1,x_2,\dots,x_d)}\left[\bigcap_{N\ge 0}\bigcup_{n\ge N}A_{\theta_n}  \right]\nonumber\\
&=&
    \lim_{N\rightarrow +\infty} \po^{(x_1,x_2,\dots,x_d)}\left[\bigcup_{n\ge N}A_{\theta_n}\right]\nonumber\\
&\ge&
    \limsup_{n\rightarrow +\infty}
    \po^{(x_1,x_2,\dots,x_d)}\big[A_{\theta_n}\big]
\ge
    C_9>0\, .
\label{MINO}
\end{eqnarray}
Observe moreover that, for a given $\omega$, $\big(S_n^{(1)},...,S_n^{(d)}\big)_n$ is a Markov chain and that $\mathbf 1_E$
is measurable, bounded and stationary (i.e. shift-invariant). Therefore,
due to a result of Doob (see e.g. \cite[Proposition V-2.4]{Neveu}),
\begin{equation}\label{Doob}
    \forall (x_1,...,x_d)\in\mathbb Z^d,
\qquad
   \lim_{n\rightarrow +\infty}P_\omega^{(S_n^{(1)},...,S_n^{(d)})}\big[E\big]
=
    \mathbf 1_E
\qquad
    P^{(x_1,...,x_d)}_\omega-\mbox{almost surely}.
\end{equation}
Also, for every $(x_1,...,x_d)\in\mathcal D$, $P_\omega^{(x_1,...,x_d)}$-almost surely, for every $n\ge 1$, $\big(S_n^{(1)},...,S_n^{(d)}\big)\in \mathcal D$. Due to
\eqref{MINO}, we obtain that, for $\p$-almost every $\omega$,
for every $(x_1,...,x_d)\in\mathcal D$, $P_\omega^{(S_n^{(1)},...,S_n^{(d)})}[{E}]\ge C_9$ for every $n$, $P_\omega^{(x_1,...,x_d)}$-almost surely.
Combining this with \eqref{Doob}, we conclude that for $\p$-almost every $\omega$,
for every $(x_1,...,x_d)\in\mathcal D$, $P_\omega^{(x_1,...,x_d)}$-a.s.,
$\mathbf 1_E\geq C_9>0$ i.e. $\mathbf 1_E=1$, and so that $P_\omega^{(x_1,...,x_d)}(E)=1$.
This ends the proof of Theorem \ref{Theorem1Meeting}.
\end{proof}

%


\section{Appendix : about random walks conditioned to stay positive}\label{Sect_RW_conditionned}

The following proposition gives the law of the potential at the left of $m_1(h)$,
and more precisely between ${}^{\uparrow}T_h(h)$ and $m_1(h)$ (defined in \eqref{eq_def_m1} and \eqref{eq_def_Tfleche_Gauche})
when ${}^{\uparrow}T_h(h)\geq 0$.
It is maybe already known, however we did not find it in the literature
(in which $m_1(h)$ is generally replaced by the local minimum of $V$ before a deterministic time instead of our
stopping time $T^{\uparrow}(h)$, see e.g. Bertoin \cite{Bertoin_Splitting}).

\begin{prop}\label{Lemma_Egalite_Loi_Gauche_de_m1}
Let $V$ be a random walk given as in \eqref{eq_def_V}
by the sequence of partial sums of i.i.d.
r.v. $\log \rho_i$, $i\in\Z$, such that $\p[\log \rho_0>0]>0$ and $\p[\log \rho_0<0]>0$
(this result does not require Hypotheses \eqref{UE}, \eqref{rho<0} or \eqref{ConditionA}).
Recall $V(-\cdot)$ and $T_{V(-\cdot)}$ from \eqref{defVminus} and \eqref{eq_Def_TV-}.
Let $h>0$. Then the process
$\big(V\big[m_1(h)-k\big]-V\big[m_1(h)\big], \, 0\leq k\leq m_1(h)-{}^{\uparrow}T_h(h)\big)$
conditioned on $\{{}^{\uparrow}T_h(h)\geq 0\}$
has the same law as
$\big(V(-k), \, 0\leq k\leq T_{V(-\cdot)}([h,+\infty[)\big)$
conditioned on $\{T_{V(-\cdot)}([h,+\infty[)<T_{V(-\cdot)}(]-\infty,0])\}$.
\end{prop}
\begin{proof}
First, notice that due to our hypotheses,
$m_1(h)<T^{\uparrow}(h)<\infty$ $\p$-almost surely, and $\p\big[{}^{\uparrow}T_h(h)\geq 0\big]>0$.
We denote by $ \bigsqcup$ the disjoint union. We write $\mathbb N^*$ for $\mathbb N\setminus \{0\}$.
Let $\psi:\bigsqcup_{t\in\mathbb N^*}\mathbb R^t\rightarrow \mathbb [0,+\infty[$ be a nonnegative measurable function with respect to the $\sigma$-algebra $\{\bigsqcup_{t\in\mathbb N^*}A_t\, :\, \forall t\in\mathbb N^*,\ A_t\in\mathcal B(\mathbb R^t)\}$.
In this proof, to simplify the notation, we set $m_1:=m_1(h)$ and ${}^{\uparrow}T:={}^{\uparrow}T_h(h)$.
We have,
\begin{eqnarray}
&&
    \E\big[
    \psi\big(V(m_1-k)-V(m_1), \, 0\leq k\leq m_1-{}^{\uparrow}T\big)
    \un_{\{{}^{\uparrow}T\geq 0\}}
    \big]
\nonumber\\
& = &
    \sum_{\ell=0}^\infty \sum_{u=1}^\infty
    \E\big[
    \psi\big(V(\ell+u-k)-V(\ell+u), \, 0\leq k\leq u\big)
    \un_{\{{}^{\uparrow}T=\ell,\ m_1=\ell+u\}}
    \big]
\nonumber\\
& = &
    \sum_{\ell=0}^\infty \sum_{u=1}^\infty
    \E\Big[
    \psi\big(V(\ell+u-k)-V(\ell+u), \, 0\leq k\leq u\big),
    \inf_{[0, \ell+u[}V > V(\ell+u),
\nonumber\\
&&
    \forall 0\leq i< \ell+u, V(i)-\inf\nolimits_{[0,i]}V<h,
    T_{V_{\ell+u}}([h,+\infty[)<T_{V_{\ell+u}}(]-\infty, 0[),
\nonumber\\
&&
    \max_{]\ell,\ell+u[}V<V(\ell+u)+h\leq V(\ell)
    \Big],
\label{eq_Loi_Lm_1}
\end{eqnarray}
where $V_{\ell+u}$ denotes the process $(V(\ell+u+x)-V(\ell+u), \,x\geq 0)$,
since the first $3$ conditions in the expectation in \eqref{eq_Loi_Lm_1} mean $m_1=\ell+u$ whereas the last one means ${}^{\uparrow}T=\ell$.
Hence,
\begin{eqnarray}
    \eqref{eq_Loi_Lm_1}
& = &
    \sum_{\ell=0}^\infty \sum_{u=1}^\infty
    \E\Big[
    \psi\big(V(\ell+u-k)-V(\ell+u), \, 0\leq k\leq u\big),
    \inf_{[\ell, \ell+u[}V > V(\ell+u),
\nonumber\\
&&
    \forall 0\leq i\leq \ell, V(i)-\inf\nolimits_{[0,i]}V<h,
    T_{V_{\ell+u}}([h,+\infty[)<T_{V_{\ell+u}}(]-\infty, 0[),
\nonumber\\
&&
    \max_{]\ell,\ell+u[}V<V(\ell+u)+h\leq V(\ell)
    \Big]
\label{eq_63}
\\
& = &
    \Pi_1(\psi)\Pi_2,
\label{eq_def_Pi1_Pi2}
\end{eqnarray}
where we used $\inf_{[0,\ell]}V>V(\ell)-h\geq V(\ell+u)$
and $V(i)-\inf_{[0,i]}V \leq \max_{]\ell,\ell+u[}V-V(\ell+u)<h$
for $\ell<i<\ell+u$
to prove that $\eqref{eq_63} = \eqref{eq_Loi_Lm_1}$ (in particular for proving $\eqref{eq_63} \leq \eqref{eq_Loi_Lm_1}$), followed by
the Markov property applied at times $\ell$ and $\ell+u$ in the last equality,
where
\begin{eqnarray*}
    \Pi_1(\psi)
& := &
    \sum_{u=1}^\infty
    \E\Big[
    \psi\big(V(u-k)-V(u), \, 0\leq k\leq u\big),
    \inf_{[0, u[}V > V(u),
    \max_{]0,u[}V<V(u)+h\leq 0
    \Big],
\nonumber\\
&&
\nonumber\\
    \Pi_2
& := &
    \sum_{\ell=0}^\infty
    \p\big[\forall 0\leq i\leq \ell, V(i)-\inf\nolimits_{[0,i]}V<h\big]
    \p\big[T_V([h,+\infty[)<T_V(]-\infty, 0[)\big].
\end{eqnarray*}
Thus, since $(V(u-k)-V(u), \ k\geq 0)$ has the same law as $(V(-k), \ k\geq 0)$,
\begin{eqnarray}
&&
    \Pi_1(\psi)
\nonumber\\
& = &
    \sum_{u=1}^\infty
    \E\Big[
    \psi\big(V(-k), \, 0\leq k\leq u\big),
    \inf_{k\in]0, u]}V(-k) > 0,
    \max_{k\in]0,u[}V(-k)<h\leq V(-u)
    \Big]
\nonumber\\
& = &
    \sum_{u=1}^\infty
    \E\Big[
    \psi\big(V(-k), \, 0\leq k\leq u\big),\,
    u=T_{V(-\cdot)}([h,+\infty[)<T_{V(-\cdot)}(]-\infty,0])
    \Big]
\nonumber\\
& = &
    \E\Big[
    \psi\big(V(-k), \, 0\leq k\leq T_{V(-\cdot)}([h,+\infty[)\big),\,
    T_{V(-\cdot)}([h,+\infty[)<T_{V(-\cdot)}(]-\infty,0])
    \Big].
~~~~~~\label{eq_Valeur_Pi1}
\end{eqnarray}
Applying this and \eqref{eq_def_Pi1_Pi2} to $\psi=1$, we get
$$
    \p\big[{}^{\uparrow}T\geq 0\big]
=
    \Pi_1(1)
    \Pi_2
=
    \Pi_2
    \p\big[T_{V(-\cdot)}([h,+\infty[)<T_{V(-\cdot)}(]-\infty,0])\big],
$$
which is non zero due to our hypotheses.
This, together with \eqref{eq_Loi_Lm_1}, \eqref{eq_def_Pi1_Pi2} and \eqref{eq_Valeur_Pi1} gives
\begin{eqnarray*}
&&
    \E\big[
    \psi\big(V(m_1-k)-V(m_1), \, 0\leq k\leq m_1-{}^{\uparrow}T\big)\
    \big|\ {}^{\uparrow}T\geq 0
    \big]
\\
& = &
    \E\Big[
    \psi\big(V(-k), \, 0\leq k\leq T_{V(-\cdot)}([h,+\infty[)\big)\
    \big|\ T_{V(-\cdot)}([h,+\infty[)<T_{V(-\cdot)}(]-\infty,0])
    \Big].
\end{eqnarray*}
This proves the lemma.
\end{proof}

\medskip

Moreover, the following proposition is useful in the proofs of Lemma  \ref{LemmaConstantInvariantMeasure} and Proposition \ref{BorelCantelli2}.
Notice in particular that we get an identity in law with a random walk conditioned to hit $[h,+\infty[$ before $]-\infty, 0[$, instead of
$]-\infty, 0]$ in the previous proposition.

\begin{prop}\label{LemmaTanakaStyleForRW}
Let $h>0$. Under the same hypotheses as in Proposition \ref{Lemma_Egalite_Loi_Gauche_de_m1},
if moreover $\liminf_{x\to+\infty}V(x)=-\infty$,\\
\noindent{\bf (i)}
The processes
$\big(V\big[m_{1}(h)-k\big]-V\big[m_{1}(h)\big],\ 0\leq k \leq m_{1}(h)\big)$
and
$\big(V\big[m_{1}(h)+k\big]-V\big[m_{1}(h)\big],\ 0\leq k \leq T^{\uparrow}(h)-m_{1}(h)\big)$
are independent.\\
\noindent{\bf (ii)}
The process
$\big(V\big[m_{1}(h)+k\big]-V\big[m_{1}(h)\big],\ 0\leq k \leq T^{\uparrow}(h)-m_{1}(h)\big)$
is equal in law to
$\big(V(k),\ 0\leq k\leq T_V([h,+\infty[) \big)$
conditioned on $\{T_V([h,+\infty[)<T_V(]-\infty,0[)\}$.
\end{prop}

\begin{proof}
We write $\mathbb N^*$ for $\mathbb N\setminus \{0\}$. Let $\psi_1$ and $\psi_2$ be two nonnegative functions,
$\bigsqcup_{t\in\mathbb N^*}\mathbb R^t\rightarrow \mathbb [0,+\infty[$, measurable with respect to the $\sigma$-algebra $\{\bigsqcup_{t\in\mathbb N^*}A_t\, :\, \forall t\in\mathbb N^*,\ A_t\in\mathcal B(\mathbb R^t)\}$.
To simplify the notation, we set $m_1:=m_1(h)$ and $T{}^{\uparrow}:=T{}^{\uparrow}(h)$.

We also recall $e_i'$ and $H_i'$, $i\in\N$, from \eqref{eq_def_ei_Prime} and \eqref{eq_def_Hi_Prime}.
Notice in particular that
$m_1=e_L'$, where
$L:=\min\{\ell\geq 0,\ H_\ell'\geq h\}<\infty$ a.s.
Hence, summing over the values of $L$, we get
\begin{align*}
&
    \E\left[\psi_1\big(V(m_1-k)-V(m_1),\ 0\leq k \leq m_1\big)
            \psi_2\big(V(m_1+k)-V(m_1),\ 0\leq k \leq T{}^{\uparrow}-m_1\big)\right]
\\
& =
    \sum_{\ell=0}^\infty
    \E\big[\psi_1\big(V\big(e_\ell'-k\big)-V\big(e_\ell'\big),\ 0\leq k \leq e_\ell'\big)\un_{\cap_{i=0}^{\ell-1}\{H_i'<h\}}
    \un_{\{H_\ell'\geq h\}}
\\
&
    \qquad\qquad
    \times
    \psi_2\big(V\big(e_\ell'+k\big)-V\big(e_\ell'\big),\ 0\leq k \leq T_{V(\cdot+e_\ell')-V(e_\ell')}([h,+\infty[)\big)\big]
\\
& =
    \Pi_3\Pi_4,
\end{align*}
by the strong Markov property at stopping time $e_\ell'$, where
\begin{align*}
    \Pi_3
& :=
    \sum_{\ell=0}^\infty
    \E\big[\psi_1\big(V\big(e_\ell'-k\big)-V\big(e_\ell'\big),\ 0\leq k \leq e_\ell'\big)\un_{\cap_{i=0}^{\ell-1}\{H_i'<h\}}
    \big]
    \p[H_\ell'\geq h]
\\
& =
    \sum_{\ell=0}^\infty
    \E\big[\psi_1\big(V\big(e_\ell'-k\big)-V\big(e_\ell'\big),\ 0\leq k \leq e_\ell'\big)\un_{\{L=\ell\}}
    \big]
\\
& =
    \E\big[\psi_1\big(V(m_1-k)
    -V(m_1),\ 0\leq k \leq m_1\big)\big]
\end{align*}
and, since $\p[H_\ell'\geq h]=\p[T_V([h,+\infty[)<T_V(]-\infty,0[)]$,
\begin{align*}
    \Pi_4
& :=
    \E\big[\psi_2\big(V(k),\ 0\leq k \leq T_V([h,+\infty[))
     \mid T_V([h,+\infty[)<T_V(]-\infty,0[)
     \big].
\end{align*}
Since this is true for any $\psi_1$ and $\psi_2$, this proves the proposition.
\end{proof}

\bigskip
\noindent{\bf Acknowledgments.}
A part of this work was done during visits to Brest and Versailles universities and during an ANR meeting.
We thank the IUF, ANR MALIN (ANR-16-CE93-0003) and the LMBA, University of Brest, the LMV, University of Versailles for their hospitality. We also thank two anonymous referees for reading very carefully the first version and giving numerous constructive comments.


\bibliographystyle{alpha}

\end{document}